\documentclass[11pt]{article}
\usepackage{color}

\usepackage{amssymb}   %basic
\usepackage{amsthm}    %theoremstyle
\usepackage{amsmath}   %\mathcal
\usepackage{stmaryrd}  %\interleave
\usepackage{titletoc}  %\ttl
\usepackage{mathrsfs}  %\mathscr
\usepackage{graphicx}

\newlength{\defbaselineskip}
\newcommand{\setlinespacing}[1]%
           {\setlength{\baselineskip}{#1 \defbaselineskip}}

\theoremstyle{plain}
\newtheorem{thm}{Theorem}[section]

\newtheorem{lem}[thm]{Lemma}
\newtheorem{prop}[thm]{Proposition}
\newtheorem{exam}[thm]{Example}

\theoremstyle{definition}
\newtheorem{defn}{Definition}[section]

\newtheorem{rmk}{Remark}[section]

\newcommand{\eps}{\varepsilon}

\DeclareMathOperator*{\esssup}{esssup}
\DeclareMathOperator*{\essinf}{essinf}

\newcommand{\cL}{\mathcal{L}}
\newcommand{\cT}{\mathcal{T}}
\newcommand{\cM}{\mathcal{M}}
\newcommand{\cB}{\mathcal{B}}
\newcommand{\cA}{\mathcal{A}}
\newcommand{\cS}{\mathcal{S}}

\newcommand{\cG}{\mathcal{G}}

\newcommand{\cU}{\mathcal{U}}

\newcommand{\bH}{\mathbb{H}}
\newcommand{\bP}{\mathbb{P}}
\newcommand{\bR}{\mathbb{R}}
\newcommand{\bN}{\mathbb{N}}

\newcommand{\sF}{\mathscr{F}}
\newcommand{\sP}{\mathscr{P}}

\makeatletter\@addtoreset{equation}{section} \makeatother
 \allowdisplaybreaks
\begin{document}

\title{Controlled Ordinary Differential Equations with Random Path-Dependent Coefficients and Stochastic Path-Dependent Hamilton-Jacobi Equations\footnotemark[1] 
}

\author{Jinniao Qiu\footnotemark[2] % \and Wenning Wei\footnotemark[2]
}
%\date{}
\footnotetext[1]{This work was partially supported by the National Science and Engineering Research Council of Canada (NSERC) and by the start-up funds from the University of Calgary. }
\footnotetext[2]{Department of Mathematics \& Statistics, University of Calgary, 2500 University Drive NW, Calgary, AB T2N 1N4, Canada. \textit{E-mail}: \texttt{jinniao.qiu@ucalgary.ca} (J. Qiu)%,  \texttt{wenning.wei@ucalgary.ca} (W. Wei)
.}
%\footnotetext[3]{Department of Mathematics \& Statistics, University of Calgary, 2500 University Drive NW, Calgary, AB T2N 1N4, Canada. \textit{E-mail}: \texttt{wenning.wei@ucalgary.ca}. }
%\footnotetext[1]{Supported by NSFC Grant \#10325101, by Basic Research Program of China (973 Program)  Grant \# 2007CB814904, by the Science
%Foundation of the Ministry of Education of China Grant \#200900071110001, and by WCU (World Class University) Program through the Korea
%Science and Engineering Foundation funded by the Ministry of Education, Science and Technology (R31-2009-000-20007).}
%
%\footnotetext[2]{Department of Finance and Control Sciences, School of Mathematical Sciences, Fudan University, Shanghai 200433, China.
%\textit{E-mail}: \texttt{071018032@fudan.edu.cn} (Jinniao Qiu), \texttt{sjtang@fudan.edu.cn} (Shanjian Tang).}
%
%\footnotetext[3]{Graduate Department of Financial Engineering, Ajou University, San 5, Woncheon-dong, Yeongtong-gu, Suwon, 443-749, Korea.}

\maketitle

%----------------------
\begin{abstract}
This paper is devoted to the stochastic optimal control problem of ordinary differential equations allowing for both path-dependence and measurable randomness. As opposed to the deterministic path-dependent cases, the value function turns out to be a random field on the path space and it is characterized by a stochastic path-dependent Hamilton-Jacobi (SPHJ) equation. A notion of viscosity solution is proposed and the value function is proved to be the unique viscosity solution to the associated SPHJ equation.
\end{abstract}



\section{Introduction}
Let $(\Omega,\sF,\{\sF_t\}_{t\geq0},\bP)$ be a complete filtered probability space. The filtration $\{\sF_t\}_{t\geq0}$ satisfies the usual conditions and is generated by an $m$-dimensional Wiener process $W=\{W(t):t\in[0,\infty)\}$ together with all the $\bP$-null sets in $\sF$. The associated predictable $\sigma$-algebra  on $\Omega\times[0,T]$ is denoted by  $\sP$. 

Throughout this work,  the number $T\in (0,\infty)$ denotes a fixed deterministic terminal time and the set $C([0,T];\bR^d)$ represents the space of $\bR^d$-valued continuous functions on $[0,T]$.  For each $x\in C([0,T];\bR^d)$, denote by $x_t$ its restriction to time interval $[0,t]$ for each $t\in [0,T]$ and by $x(t)$ its value in $\bR^d$ at time $t\in[0,T]$. Consider the following stochastic optimal control problem
\begin{align}
\min_{\theta\in\cU}E\left[\int_0^T\!\! f(s,X_s,\theta(s))\,ds +G(X_T) \right], \label{Control-probm}
\end{align}
subject to
\begin{equation}\label{state-proces-contrl}
\left\{
\begin{split}
&\frac{dX(t)}{dt}=\beta(t,X_t,\theta(t)),\,\,
\,t\geq 0; \\
& X_0=x_0\in\bR^d.
\end{split}
\right.
\end{equation}
% where $T\in (0,\infty)$ is a fixed deterministic terminal time, and $X(t)$ and $\theta(t)$ denote the values at time $t$ while $X_t=(X(s),\,s\in [0,t])$ represents the path up to time $t$.
Here, we denote by $\cU$ the set of all the $U$-valued and $\sF_t$-adapted processes with $U\subset \bR^{\bar m}$ ($\bar m\in\bN^+$) being a nonempty set. The process $(X(t))_{t\in[0,T]}$ is the {\sl state process}, governed by the {\sl control} $\theta\in\cU$. We may write $X^{r,x_r;\theta}(t)$ for $0\leq r\leq t\leq T$ to indicate the dependence of the state process on the control $\theta$, the initial time $r$ and initial path $x_r$.  

In this paper, we consider the non-Markovian cases where the coefficients $\beta,f$, and $G$ may depend not only on time and control but also \textit{explicitly} on $\omega \in\Omega$ and paths/history of the state process. 
Inspired by the (Markovian) counterparts in \cite{ekren2016viscosity-1,lukoyanov2007viscosity},  such problems may arise naturally from controlled ordinary differential equations allowing for path-dependent and random coefficients. Two interesting examples are sketched as follows.
%Particular examples include the swing option valuation with general underlying payoff processes \cite{bender2016-first-order} and the following
\begin{exam}
In terms of the optimization \eqref{Control-probm}-\eqref{state-proces-contrl}, we may study the control problem of some general stochastic equations:
\begin{align}
\min_{\theta\in\cU}E\left[\int_0^T\!\!  \tilde f(s,\tilde X_s,\theta(s))\,ds + \tilde G(\tilde X_T) \right] \label{Control-probm-exam}
\end{align}
subject to
\begin{equation}\label{state-proces-contrl-exam}
%\left\{
%\begin{split}
 {\tilde X(t)}=  \tilde x_0 + \int_0^t \tilde \beta(s,\tilde X_s,\theta(s))\,ds+\eta(t),\,\,
\,t\geq 0, \tilde X_0= x_0\in\bR^d,
%\end{split}
%\right.
\end{equation}
where the new term $\left( \eta(t)\right)_{t\geq 0}$ may be any general $\sF_t$-adapted stochastic process with or without rough paths, including Wiener processes, fractional Brownian motions, and associated integrals, etc.
%where the new coefficient $\sigma$: $\Omega\times[0,T]\rightarrow \bR^{d\times m}$ may be generally an $\sF_t$-adapted square-integrable process. 
Set $X(t)=\tilde X(t)-\eta(t)$, for $t\in[0,T]$. The control problem \eqref{Control-probm-exam}-\eqref{state-proces-contrl-exam} may be written equivalently as \eqref{Control-probm}-\eqref{state-proces-contrl}, while the associated coefficients $(f,\beta)(t, X_t,\theta(t))=(\tilde f,\tilde \beta)(s, (X+\eta)_s,\theta(s))$ and $G(X_T)=\tilde G((X+\eta)_T)$ may of course be random. 
\end{exam}
\begin{exam}[Optimal consumption with habit formation] ~\\
Adopting the von Neumann-Morgenstern preferences over time interval $[0,T]$, the utility maximization with habit formation may be generally formulated as an optimal stochastic control problem like \eqref{Control-probm}-\eqref{state-proces-contrl}:
\begin{align*}
\max_{c\in \tilde \cA} E \int_0^T u(t,c(t),Z(t,C_t))\,dt
\end{align*}
subject to
\begin{equation*} 
\left\{
\begin{split}
dC(t) &=  c(t)\,dt, \quad t\in [0,T]; \quad C(0)=\kappa;\\
Z(t,C_t)&=\gamma(t)+\int_0^t\xi(t,s)\,dC(s), \quad t\in[0,T],
\end{split}
\right.
\end{equation*}
where we denote by $(c(t))_{t\in [0,T]}$ the consumption process with $C(t)$ the cumulative consumption until time $t$ and the process $Z(t,C_t)$  represents the standard of living with random coefficients $\gamma(t)$ and $\xi(t,s)$ for $0\leq s\leq t\leq T$. The admissible control set $\tilde \cA$ may be complicated by incorporating the budget constraints for different market models; refer to \cite{detemple1991asset,EnglezosKaratzas09,yu2017optimal} among many others.
\end{exam}

Coming back to the control problem \eqref{Control-probm}-\eqref{state-proces-contrl}, we define the dynamic cost functional:\begin{align}
J(t,x_t;\theta)=E_{\sF_t}\left[\int_t^T\!\! f(s,X^{t,x_t;\theta}_s,\theta(s))\,ds +G(X^{t,x_t;\theta}_T) \right],\ \ t\in[0,T], \label{eq-cost-funct}
\end{align}
and the value function is given by
\begin{align}
V(t,x_t)=\essinf_{\theta\in\cU}J(t,x_t;\theta),\quad t\in[0,T].
\label{eq-value-func}
\end{align}
Here and throughout this paper, we denote by $E_{\sF_t}[\,\cdot\,]$ the conditional expectation with respect to $\sF_t$. Due to the randomness and path-dependence of the coefficient(s), the value function $V(t,x_t)$ is generally a function of time $t$, path $x_t$, and $\omega\in\Omega$. In fact, the value function $V$ is proved to be the unique viscosity solution to the following stochastic path-dependent Hamilton-Jacobi (SPHJ) equation:
\begin{equation}\label{SHJB}
  \left\{\begin{array}{l}
  \begin{split}
  -\mathfrak{d}_t u(t,x_t)
- \mathbb{H}(t,x_t,\nabla u(t,x_t) )&=0,\quad (t,x)\in [0,T)\times  C([0,T];\bR^d);\\
    u(T,x)&= G(x), \quad x\in   C([0,T];\bR^d),
    \end{split}
  \end{array}\right.
\end{equation}
with 
\begin{align*}
\mathbb{H}(t,x_t,p)
= \inf_{v\in U} \bigg\{
       \beta'(t,x_t,v)p +f(t,x_t,v)
                \bigg\},\quad \text{for } p\in \bR^d,  
\end{align*}
  where $\nabla u(t,x_t)$ denotes the the vertical derivative of $u(t,x_t)$ at the path $x_t$ (see Definition \ref{defn-VD}) and  the unknown \textit{adapted} random field $u$ is confined to the following form:
  \begin{align}
u(t,x_t)=u(T,x_{t,T-t})-\int_{t}^T\mathfrak{d}_{s} u(s,x_{t,s-t})\,ds-\int_t^T\mathfrak{d}_{w}u(s,x_{t,s-t})\,dW(s), \label{SDE-u}
\end{align}
where $x_{t,r-t}(s)=x_t(s) {1}_{[0,t)}(s) + x_t(t){1}_{[t,r]}(s)$ for $0\leq t\leq s\leq r\leq T$.
The Doob-Meyer decomposition theorem indicates the uniqueness of the pair $(\mathfrak{d}_tu,\,\mathfrak{d}_{\omega}u)$ and thus the linear operators $\mathfrak{d}_t$ and $\mathfrak{d}_{\omega}$ may be well defined in certain spaces (see Definition \ref{defn-testfunc}).  The pair $(\mathfrak{d}_tu,\,\mathfrak{d}_{\omega}u)$ may also be defined as two differential operators; see \cite[Section 5.2]{cont2013-founctional-aop} and \cite[Theorem 4.3]{Leao-etal-2018} for instance. By comparing \eqref{SHJB} and \eqref{SDE-u}, we may rewrite the SPHJ equation \eqref{SHJB} formally as a path-dependent backward stochastic partial differential equation (BSPDE): 
\begin{equation}\label{SHJB-equiv}
  \left\{\begin{array}{l}
  \begin{split}
  -du(t,x_t)=\,& %\frac{1}{2}\text{tr}\left\{\left(+ D^2u(t,x)\right)\right\}
 \mathbb{H}(t,x_t,\nabla u(t,x_t)) 
 \,dt -\psi(t,x_t)\, dW({t}), \quad
                     (t,x)\in [0,T)\times  C([0,T];\bR^d);\\
    u(T,x)=\, &G(x), \quad x\in C([0,T];\bR^d),
    \end{split}
  \end{array}\right.
\end{equation}
where the pair $(u,\psi)=(u,\mathfrak{d}_{w}u)$ is unknown. 

When all the coefficients are \textit{deterministic} path-dependent functions, the control problem is first studied in \cite{lukoyanov2007viscosity} and  the viscosity solution theory of associated \textit{deterministic} path-dependent Hamilton-Jacobi equations may be found in \cite{ekren2016viscosity-1,lukoyanov2007viscosity}; for the theory of general deterministic path-dependent PDEs, see \cite{bayraktar2018path,cont2013-founctional-aop,cosso2018path,peng2016bsde,ren2017comparison} among many others. When all the coefficients are just possibly random but with \textit{state}-dependence, the resulting Hamilton-Jacobi equation is just a BSPDE (see \cite{bender2016-first-order,qiu2017viscosity,qiu2019uniqueness}); for related research on general BSPDEs, we refer to \cite{Bayraktar-Qiu_2017,cardaliaguet2015master,DuTangZhang-2013,Hu_Ma_Yong02,Peng_92} to mention but a few. When the coefficients are \textit{both} random and path-dependent, some discussions may be found in \cite{ekren2016viscosity-1,ekren2016pseudo,ren2017comparison} where, nevertheless, all the coefficients are required to be continuous and deterministic in $\omega\in  \Omega$ and the resulting value function satisfies, instead, a deterministic path-dependet semilinear parabolic PDE.  In the present work, all the involved coefficients are just measurable w.r.t. $\omega\in \Omega$ without any topology specified on $\Omega$, which allows the general random variables to appear in the coefficients, and in this setting, the control problem \eqref{Control-probm}-\eqref{state-proces-contrl} and associated nonlinear SPHJ equation \eqref{SHJB}, to the best of our knowledge, have never been studied in the literature. In this paper, we are the first to use the dynamic programming method to deal with the control problem \eqref{Control-probm}-\eqref{state-proces-contrl} allowing for both path-dependence and \textit{measurable} randomness; a notion of viscosity solution is proposed and the existence and uniqueness of viscosity solution is proved under standard Lipschitz conditions.

The main feature of viscosity solution for SPHJ equation \eqref{SHJB} is twofold. On the one hand, due to the path-dependence of the coefficients, the solution $u$ is path-wisely defined on the path space $C([0,T];\bR^d)$, while the lack of local compactness of the path space prompts us to define the random test functions in certain compact subspaces via optimal stopping times, which is different from the capacity and nonlinear expectation techniques for deterministic path-dependent PDEs (see \cite{ekren2014viscosity,ekren2016viscosity-1,ren2014overview} for instance). On the other hand,  as the involved coefficients are  just measurable w.r.t. $\omega$ on the sample space $(\Omega,\sF)$ without any equipped topology,  it is not appropriate to define the viscosity solutions in a pointwise manner w.r.t. $\omega\in (\Omega,\sF)$; rather, we use a class of random fields of form \eqref{SDE-u} having sufficient spacial regularity as test functions; at each point $(\tau,\xi)$ ($\tau$ may be stopping time and $\xi$ a $C([0,\tau];\bR^d)$-valued $\sF_{\tau}$-measurable variable) the classes of test functions are also parameterized by $\Omega_{\tau}\in\sF_{\tau}$ and the type of compact subspaces. It is worth noting that due to the  nonanticipativity constraint on the unknown function, the conventional variable-doubling techniques for deterministic Hamilton-Jacobi equations are not applicable in our stochastic setting. Instead, we prove a \textit{weak} comparison principle with comparison conducted on compact subspaces, and then the uniqueness of viscosity solution is proved through regular approximations.

The rest of this paper is organized as follows. In Section 2, we introduce some notations, show the standing assumption on the coefficients, and define the viscosity solution. In Section 3, some auxiliary results including the dynamic programming principle are given in the first subsection and  the value function is verified to be a viscosity soution in the second subsection. In Section 4, a weak comparison theorem is given and then the uniqueness is  proved via approximations. Finally, we give the proof of Theorem \ref{thm-DPP} in the appendix.

%%%%%%%%%%%%%%%%%%%%%%%%%%%%%%%%%%%%%%%%%%%%%%%%%%%%%%%%%%%%%%%%%%%%%%%%%%%%%%%%%%%%%%
%%%%%%%%%%%%%%%%%%%%%%%%%%%%%%%%%%%%%%%%%%%%%%%%%%%%%%%%%%%%%%%%%%%%%%%%%%%%%%%%%%%%%
%%%%%%%%%%%%%%%%%%%%%%%%%%%%%%%%%%%%%%%%%%%%%%%%%%%%%%%%%%%%%%%%%%%%%%%%%%%%%%%%%%%%%%%%%%%%%%%%%%%%%%%%%%%%%%%%%%%%%%%%%%%%%%%%%%%%%%%%%%%%%%%%%%%%%%%%%%%%%%%%%%%%%%%%%%%%%%%%%%%%

\section{Preliminaries and definition of viscosity solution}

\subsection{Preliminaries}
%Denote by $|\cdot|$ the norm in  Euclidean spaces.
%%For each $l\in \mathbb{N}^+$ and domain $\Pi\subset \bR^l$, denote by $C_c^{\infty}(\Pi)$ the space of infinitely differentiable functions with compact supports in $\Pi$. 
% Define the parabolic distance in $\bR^{1+d}$ as follows:
%$$\delta(X,Y):=\max\{ |t-s|^{1/2},|x-y| \},$$
% for $X:=(t,x)$ and $Y:=(s,y)\in \bR^{1+d}$. Denote by $Q^+_r(X)$ the hemisphere of radius $r>0$ and center $X:=(t,x)\in \bR^{1+d}$ with $x\in \bR^d$:
%\begin{equation*}
%  \begin{split}
%    Q^+_r(X):=\,  [t,t+r^2)\times B_r(x), \quad
%    B_r(x):=\, \{ y\in\bR^n:|y-x|<r \},
%  \end{split}
%\end{equation*}
%and by $|Q^+_r(X)|$ the volume. 

%For each $r\in [0,T]$, both the continuous function space $C([0,r];\bR^d)$ and the space of cadlag (right continuous with left limits) functions denoted by $D([0,r];\bR^d)$ are equipped with the uniform norm:
%$$\|X\|_{r}:=\max_{t\in[0,r]} |X(t)|,\quad \text{for } X\in C([0,r];\bR^d) \text{ or } X\in D([0,r];\bR^d).$$
%Then both spaces $C([0,r];\bR^d)$ and $D([0,r];\bR^d)$ are complete under the norm $\|\cdot\|_{r}$.

   For each integer $k>0$ and $r\in[0,T]$, let space $\Lambda_r^0(\bR^k):=C([0,r];\bR^k)$ be the set of all $\bR^k$-valued continuous functions on $[0,r]$ and $\Lambda_r(\bR^k):=D([0,r];\bR^k)$  the space of $\bR^k$-valued c\`{a}dl\`ag (right continuous with left limits) functions on $[0,r]$. Define
$$
\Lambda^0(\bR^k)=\cup_{r\in[0,T]} \Lambda^0_r(\bR^k),\quad \Lambda(\bR^k)=\cup_{r\in[0,T]} \Lambda_r(\bR^k).
$$
Throughout this paper, for each path $X\in \Lambda_T(\bR^k)$ and $t\in[0,T]$, denote by $X_t=(X(s))_{0\leq s\leq t}$ its restriction to time interval $[0,t]$,  while using $X(t)$ to represent its value in $\bR^k$ at time $t$, and moreover, when $k=d$, we write simply $\Lambda^0$, $\Lambda_r^0$, $\Lambda$, and $\Lambda_r$.

Both the spaces $\Lambda$ and $\Lambda^0$ are equipped with the following quasi-norm and metric: for each $(x_r,\bar x_t)\in \Lambda_r\times \Lambda_t$ or $(x_r,\bar x_t)\in  \Lambda^0_r\times \Lambda^0_t$ with $0\leq r\leq t\leq T$,
\begin{align*}
\|x_r\|_{0}&= \sup_{s\in[0,r]} |x_r(s)|;
\\
d_0(x_r,\bar x_t)&=\sqrt{|t-r|} + \sup_{s\in [0,t]} \left\{ |x_r(s)-\bar x_t(s)| 1_{[0,r)}(s)+ |x_r(r)-\bar x_t(s)| 1_{[r,t]}(s) \right\}.
\end{align*}
Then $(\Lambda^0_t, \|\cdot\|_0)$ and $(\Lambda_t, \|\cdot\|_0)$ are Banach spaces for each $t\in[0,T]$, while $(\Lambda^0,  d_0)$ and $(\Lambda, d_0)$ are complete metric spaces. In fact, for each $t\in[0,T]$, $(\Lambda^0_t, \|\cdot\|_0)$ and $(\Lambda_t, \|\cdot\|_0)$ can be and (throughout this paper) will be thought of as the complete subspaces of $(\Lambda^0_T, \|\cdot\|_0)$ and $(\Lambda_T, \|\cdot\|_0)$, respectively; indeed, for each $x_t\in \Lambda_t$ ($x_t\in \Lambda_t^0$, respectively), we may define, correspondingly, $\bar x\in\Lambda_T$ ($\bar x\in\Lambda^0_T$, respectively) with $\bar x (s)=x_t(t\wedge s)$ for $s\in[0,T]$ and throughout this work, we do not distinguish between $x$ and $\bar x$. In addition, we shall use $\cB(\Lambda^0)$, $\cB(\Lambda)$, $\cB(\Lambda^0_t)$ and $\cB(\Lambda_t)$ to denote the corresponding Borel $\sigma$-algebras. By contrast, for each $\delta>0$ and $x_r\in\Lambda$, denote by $B_{\delta}(x_r)$ the set of paths $y_t\in\Lambda$ satisfying $d_0(x_r,y_t)\leq \delta$.

 For each $x_t\in \Lambda_t$ and any $h\in\bR^d$, we define its vertical perturbation $x^h_t \in \Lambda_t$ with
 $x_t^h(s)= x_t(s) 1_{[0,t)}(s)+ \left(x_t(t)+h\right)1_{\{t\}}(s)$ for $s\in[0,t]$. 
% $$
% x_t^h(s)= x_t(s) 1_{[0,t)}(s)+ \left(x_t(t)+h\right)1_{\{t\}}(s), \quad \text{for }s\in[0,t) \quad \text{and}\quad x_t^h(t)=x_t(t)+h.
% $$
 \begin{defn}\label{defn-VD}
 Given a functional $\phi$: $\Lambda \rightarrow \bR$ and $x_t\in \Lambda_t$, $\phi$ is said to be vertically differentiable at $x_t$ if the function 
 \begin{align*}
 \phi(x_t^{\cdot}):\ \ 
 &\bR^d\rightarrow \bR,\\
&h\mapsto  \phi(x_t^h)
 \end{align*}
 %$\phi(x_t^h)$: $\bR^d\rightarrow \bR$, $h\mapsto \phi(x_t^h)$ is differentiable at $0$. 
 is differentiable at $0$. The gradient  
 $$
 \nabla \phi(x_t):=(\nabla_1\phi(x_t),\dots, \nabla_d\phi(x_t))' 
 \quad\text{with}\quad 
 \nabla_i \phi(x_t):=\lim_{\delta\rightarrow 0}\frac{\phi(x_t^{\delta e_i}) -\phi(x_t)}{\delta}
 $$
  is called the vertical derivative of $\phi$ at $x_t$, where $\{e_i\}_{i=1,\dots,d}$ is the canonical basis in $\bR^d$.
 \end{defn}

Let $\mathbb B $ be a Banach space equipped with norm $\|\cdot\|_{\mathbb B }$. The continuity of functionals on metric spaces $\Lambda^0$ and $\Lambda$ can be defined in a standard way. Given $x_t\in \Lambda$, we say a map $\phi$: $\Lambda\rightarrow \mathbb B$ is continuous at $x_t$ if for any $\eps >0$ there exists some $\delta >0$ such that for any $\bar x_r\in\Lambda$ satisfying $d_0(\bar x_r,x_t)<\delta$, it holds that $\|\phi(x_t)-\phi(\bar x_r)\|_{\mathbb B}<\eps$.  If the $\mathbb B$-valued functional $\phi$ is continuous and bounded at all $x_t\in\Lambda$, $\phi$ is said to be continuous on $\Lambda$ and denoted by $\phi\in C(\Lambda; \mathbb B)$. 
%Without the subscript, we use $C(\Lambda;\mathbb B)$ to denote the space of all the bounded elements of $C_{loc}(\Lambda;\mathbb B)$. 
Similarly, we define $C(\Lambda^0;\mathbb B)$, $C([0,T]\times\Lambda;\mathbb B)$, and $C([0,T]\times\Lambda^0;\mathbb B)$.

Throughout this work, we denote by $(\Omega,\sF,\{\sF_t\}_{t\geq0},\bP)$ a complete filtered probability space. The filtration $\{\sF_t\}_{t\geq0}$ satisfies the usual conditions and is generated by an $m$-dimensional Wiener process $W=\{W(t):t\in[0,\infty)\}$ together with all the $\bP$-null sets in $\sF$. The associated predictable $\sigma$-algebra  on $\Omega\times[0,T]$ is denoted by  $\sP$. 

 For each $t\in[0,T]$, denote by $L^0(\Omega\times\Lambda_t,\sF_t\otimes\cB(\Lambda_t);\mathbb B)$ the space of $\mathbb B$-valued $\sF_t\otimes\cB(\Lambda_t)$-measurable random variables. For each measurable function
 %$u$: $\Omega\times[0,T]\times \Lambda\rightarrow \mathbb B$, 
 $$
 u: \quad (\Omega\times[0,T]\times \Lambda,\, \sF\otimes\cB([0,T])\otimes \cB(\Lambda)  )   \rightarrow (\mathbb B,\,\cB(\mathbb B)),
 $$
 we say $u$ is \textit{adapted} if for any time $t\in [0,T]$, $u$ is $\sF_t\otimes \cB(\Lambda_t)$-measurable.  
 For $p\in[1,\infty]$, denote by $\cS ^p (\Lambda; {\mathbb B })$ the set of all the adapted functions $u$: $\Omega\times[0,T]\times \Lambda\rightarrow \mathbb B$ such
 that for almost all $\omega\in\Omega$, $u$ is valued in $C([0,T]\times\Lambda;\mathbb B)$
 %while by $\cS ^p (\Lambda; {\mathbb B })$ we denote the subset of $\cS ^p_{loc} (\Lambda; {\mathbb B })$ with each $u\in \cS ^p (\Lambda; {\mathbb B })$ being valued in $C([0,T]\times\Lambda;\mathbb B)$ 
 and
{\small $$\| u\|_{\cS ^p(\Lambda; {\mathbb B })}:= \left\|\sup_{(t,x_t)\in [0,T]\times \Lambda_t} \|u(t,x_t)\|_{\mathbb B }\right\|_{L^p(\Omega,\sF,\bP)}< \infty.$$
}
 For $p\in [1,\infty)$, denote by $\mathcal{L}^p(\Lambda; {\mathbb B })$ the totality of all  the adapted functions $\mathcal X$: $\Omega\times[0,T]\times \Lambda\rightarrow \mathbb B$ such
 that for almost all $(\omega,t)\in\Omega\times[0,T]$, $\mathcal X(t)$ is valued in $C(\Lambda_t;\mathbb B)$, and 
 $$
 \|\mathcal X\|_{\mathcal{L}^p(\Lambda; {\mathbb B })}:=\left\| \bigg(\int_0^T \sup_{x_t\in\Lambda_t}\|\mathcal X(t,x_t)\|_{\mathbb B }^p\,dt\bigg)^{1/p} \right\|_{L^p(\Omega,\sF,\bP)}< \infty.
 $$
%  By $\cM^p(V)$ we denote the space of all  the $V$-valued,
% $(\sF_t)$-adapted processes $\{\mathcal X_{t}\}_{t\in [0,T]}$ such
% that
%{\small
% $$
% \|\mathcal X\|_{\mathcal{M}^p(V)}:=\left\|  \sup_{t\in[0,T]}\bigg(E_{\sF_t} \int_t^T \|\mathcal X_t\|_V^p\,dt\bigg)^{1/p} \right\|_{L^{\infty}(\Omega,\sF,\bP)}< \infty.
% $$
% }
Obviously, $(\cS^p(\Lambda; {\mathbb B }),\,\|\cdot\|_{\cS^p(\Lambda; {\mathbb B })})$ and $(\mathcal{L}^p(\Lambda; {\mathbb B }),\|\cdot\|_{\mathcal{L}^p(\Lambda; {\mathbb B })})$
%and $(\mathcal{H},\|\cdot\|_{\mathcal{H}})$
are Banach spaces. In a standard way, we define $L^0(\Omega\times\Lambda^0_t,\sF_t\otimes\cB(\Lambda^0_t);\mathbb B)$, $(\cS^p(\Lambda^0; {\mathbb B }),\,\|\cdot\|_{\cS^p(\Lambda^0; {\mathbb B })})$, and $(\mathcal{L}^p(\Lambda^0; {\mathbb B }),\|\cdot\|_{\mathcal{L}^p(\Lambda^0; {\mathbb B })})$.
 \\

Following is the assumption we use throughout this paper.\\[5pt]
\noindent
   $({\mathcal A} 1)$ \it $G\in L^{\infty}(\Omega,\sF_T;C(\Lambda_T;\bR))$. For the coefficients $g=f,\beta^i$ $(1\leq i \leq d)$, \\
(i) 
%$g:~(\Omega\times[0,T]\times\Lambda\times U,\,\sF\otimes\cB([0,T])\otimes\cB(\Lambda)\otimes\cB(U))\rightarrow(\bR,\,\cB(\bR)) $,
 for each $v\in U$, $g(\cdot,\cdot,v)$ is adapted;\\
(ii) for almost all $(\omega,t)\in\Omega\times [0,T]$, $g(t,\cdot,\cdot)$ is continuous on $\Lambda_t\times U$;\\
(iii) there exists $L>0$ such that for all $x,\bar x\in \Lambda_T$, $t\in[0,T]$ and $\gamma_t,\bar\gamma_t\in \Lambda_t$, there hold
\begin{align*}
\esssup_{\omega\in\Omega} |G(x)|+ 
\esssup_{\omega\in\Omega} \sup_{v\in U} |g(t,\gamma_t, v) |
& \leq L ,\\
\esssup_{\omega\in\Omega} |G(x)-G(\bar x)|+ 
\esssup_{\omega\in\Omega} \sup_{v\in U} |g(t,\gamma_t, v)-g(t,\bar \gamma_t,v)|
& \leq L\left(\|x-\bar x\|_0 + \|\gamma_t-\bar\gamma_t\|_0   \right).
\end{align*}
\rm

%There exist $\alpha\in(0,1]$ and $L>0$ and some $g\in \cL^2(L^2)$   such that for all $x_1,x_2\in \bR^d$, $u_1,u_2\in\bR$, $v\in U$   and $(\omega,t)\in \Omega\times[0,T]$,
%   $ f(\omega,t,x_1,0,v)
%       \leq\,
%       g(\omega,t,x_1)$
%       and
%   \begin{align*}
%       &|f(\omega,t,x_1,u_1,v)-f(\omega,t,x_2,u_1,v)|+|G(\omega,x_1)-G(\omega,x_2)|
%       \leq\, L|x_1-x_2|^{\alpha},
%       \\
%     & |f(\omega,t,x_1,u_1,v)-f(\omega,t,x_1,u_2,v)|
%       \leq\, L|u_1-u_2|.
%   \end{align*}\rm

%%%%%%%%%%%%%%%%%%%%%%%%%%%%%%%%%%%%%%%%%%%%%%%%%%%%%%%%%%%%%%%%%%%%%%%%%%%%%%%%%%%%%%%%%%
%%%%%%%%%%%%%%%%%%%%%%%%%%%%%%%%%%%%%%%%%%%%%%%%%%%%%%%%%%%%%%%%%%%%%%%%%%%%%%%%%%%%%%%%%%%%%%%%%%%%%%%%%%%%%%%%%%%%%%%%%%%%%%%%%%%%%%%%%%%%%%%%%%%%%%%%%%%%%%%%%%%%%%%%%%%%%%%%%%%%%%%%%%%%%%%%%%%%%%%%%%%%%%%%%%%%%%%%%%%%%%%%%%%%%%%%%%%%%%%%%%%%%%%%%%%%%%%%%%%%%%%%%%%%%%%%%%%%%%%%%%%%%%%%%%%%%%%%%%%%%%%%%%%%%%%%%%%%%%%%%%%%%%%%%%%%%%%%%%%%%%%%%%%%%%%%%%%%%%%%

\subsection{Definition of viscosity solutions}

%\subsection{Definition of viscosity solutions}
%We first introduce the test function spaces for viscosity solutions.
For $\delta \geq 0$, $x_t\in \Lambda_t$, we define the horizontal extension $x_{t,\delta} \in \Lambda_{t+\delta}$ by setting $x_{t,\delta}(s)=x_t(s\wedge t)$ for all $s\in[0,t+\delta]$.

\begin{defn}\label{defn-testfunc}
For $u\in \cS^{2} (\Lambda;\bR)$ with $\nabla u\in \cL^2(\Lambda;\bR)$, we say $u\in \mathscr C_{\sF}^1$ if \\
%$u(T,\cdot)\in L^{2}(\Omega,\sF_T;C^1(\bR^d))  $ and 
(i) there exists $(\mathfrak{d}_tu, \,\mathfrak{d}_{\omega}u)\in \cL^2(\Lambda;\bR)\times  \cL^2(\Lambda;\bR^m)$ such that
%(i) with probability 1
%\begin{align*}
%u(r,x_r)=u(T,x_T)-\int_r^T \mathfrak{d}_su(s,x_s)\,ds -\int_r^T\mathfrak{d}_{\omega}u(s,x_s)\,dW(s),\quad \forall\,(r,x)\in[0,T]\times\Lambda_T;
%\end{align*}
for all $0\leq r \leq \tau \leq T$, and $x_r\in \Lambda_r$, it holds that
\begin{align}
 u(\tau,x_{r,\tau-r})=u( r,x_{r}) + \int_{r}^{\tau} 
 	\mathfrak{d}_s u(s,x_{r,s-r})\,ds 
+\int^{\tau}_r\mathfrak{d}_{\omega}  u(s,x_{r,s-r})\,dW(s),\text{ a.s.;}
\label{derivative}
\end{align}
(ii) there exists a constant $\varrho\in (0,\infty)$ such that  for almost all $(\omega,t)\in\Omega\times [0,T]$ and all $x_t\in \Lambda^0_t$, there holds $|\nabla u(t,x_t)| \leq \varrho$;\\
(iii) there exist a constant $\alpha\in (0,1)$ and a finite partition $0=\underline t_0<\underline t_1<\ldots<\underline t_n=T$, for integer $n\geq1$, such that $\nabla u$ is a.s. valued in $C((\underline t_j,\underline t_{j+1})\times \Lambda;\bR^d)$ for $j=0,\ldots,n-1$, and  for any $0<\delta<\min_{0\leq j \leq n-1} |\underline t_{j+1}-\underline t_j|$, and  $g=\mathfrak{d}_tu,\,\nabla_i u, \,(\mathfrak{d}_{\omega}u)^j$, $i=1,\dots,d$, $j=1,\dots,m$, there exists $L^{\delta}_{\alpha} \in (0,\infty) $ satisfying a.s. for almost all $t\in \cup_{0\leq j \leq n-1} (\underline t_j, \underline t_{j+1}-\delta]$ and all $x_t,y_t\in \Lambda_t$,
\begin{align}
|g(t,x_t)-g(t,y_t)| \leq L^{\delta}_{\alpha} \left\|x_t-y_t\right\|_0^{\alpha}. \label{holder-regulr}
\end{align}
 We say the number $\alpha$ is the exponent associated to $u\in \mathscr C^1_{\sF}$ and $0=\underline t_0<\underline t_1<\ldots<\underline t_n=T$ the associated partition.\footnote{The exponent $\alpha$ is not put in the notation $\mathscr C^{1}_{\sF}$, as in many applications, there is no need to specify the exponent.}
\end{defn}

Each $u\in \mathscr C_{\sF}^1 $ may be thought of as an It\^o process and thus a semi-martingale parameterized by $x\in\Lambda$. 
The uniqueness of semimartingale decomposition (by Doob-Meyer decomposition theorem) ensures the uniqueness of the pair $(\mathfrak{d}_tu,\,\mathfrak{d}_{\omega}u)$ at points $(\omega, t,x_{s,t-s})$ for $0\leq s <t\leq T$. Recall that by the definition of the space $\mathcal{L}^2(\Lambda; {\mathbb B })$ where $\mathbb B$ denotes a Banach space, for almost all $(\omega,t)\in\Omega\times [0,T]$, $g(t)$ is valued in $C(\Lambda_t;\mathbb B)$ for $g\in \mathcal{L}^p(\Lambda; {\mathbb B })$.  With a standard denseness argument we may define the pair $(\mathfrak{d}_tu,\,\mathfrak{d}_{\omega}u)$ in $ \cL^2(\Lambda;\bR)\times  \cL^2(\Lambda;\bR^m)$ with
$$(\mathfrak{d}_tu,\,\mathfrak{d}_{\omega}u)(t,x_t)=\lim_{s\rightarrow t^-}(\mathfrak{d}_tu,\,\mathfrak{d}_{\omega}u)(t,x_{s,t-s})=(\mathfrak{d}_tu,\,\mathfrak{d}_{\omega}u)(t,\lim_{s\rightarrow t^-} x_{s,t-s})\text{ a.s., }\forall\,  x_t\in \Lambda_t, $$
for almost all $t\in(0,T]$, and for each $x_0\in\bR^d$, set 
$$(\mathfrak{d}_tu,\,\mathfrak{d}_{\omega}u)(0,x_0)
=\lim_{h\rightarrow 0^+}\frac{1}{h}\int_0^h(\mathfrak{d}_su,\,\mathfrak{d}_{\omega}u)(s,x_{0,s})ds, $$ if the limit exists; otherwise put $(\mathfrak{d}_tu,\,\mathfrak{d}_{\omega}u)(0,x_0)=(0,0)$.  This makes sense of the two linear operators $\mathfrak{d}_t$ and $\mathfrak{d}_{\omega}$ which are consistent with the differential operators  in \cite[Section 5.2]{cont2013-founctional-aop} and \cite[Theorem 4.3]{Leao-etal-2018}. For the reader's interest, we may consider an example when $m=d=1$:
$$u(t,x_t)=\sin\left( \int_0^tg(s,t)\cos\left(x(s)+\|W_s\|_0\right) ds + x(t)+2W(t)\right),\quad t\in[0,T],\quad x_t\in \Lambda(\bR),$$
with $g\in C^2([0,T]\times[0,T])$. Then we have $u\in \mathscr C_{\sF}^1$
\begin{align*}
\nabla u(t,x_t)&=\cos\left( \int_0^tg(s,t)\cos\left(x(s)+\|W_s\|_0\right)  ds + x(t)+2W(t)\right),
\\
\mathfrak{d}_tu(t,x_t)&= -2u(t,x_{t-})+
 \cos\left( \int_0^tg(s,t)\cos\left(x(s)+\|W_s\|_0\right)  ds + x(t-)+2W(t)\right) 
  \\
 &\quad\quad 
  \cdot \bigg[\int_0^t\frac{\partial g(s,t)}{\partial t} \cos\left(x(s)+\|W_s\|_0\right) ds +g(t,t)\cos\left(x(t-)+\|W_t\|_0\right) \bigg] ,
 \\
 \mathfrak{d}_{\omega}u(t,x_t)&=
 2 \cos\left( \int_0^tg(s,t)\cos\left(x(s)+\|W_s\|_0\right)  ds + x(t-)+2W(t)\right),
 \end{align*}
where $x(0-)=x(0)$, $x(t-)=\lim_{s\rightarrow t-}x(s)$ for $t\in(0,T]$ and $x_{t-}=x_t^{x(t-)-x(t)}$.

Particularly, if $u(t,x)$ is a deterministic function on the time-\textit{state} space $[0,T]\times\bR^d$, we may have $\mathfrak{d}_{\omega}u \equiv 0$ and $\mathfrak{d}_tu$ coincides with the classical partial derivative in time; if the random function $u$ on $\Omega\times[0,T]\times \bR^d$ is regular enough (w.r.t. $\omega$), the term $\mathfrak{d}_{\omega}u$ is just the Malliavin derivative.  In addition, the operators $\mathfrak{d}_t$ and $\mathfrak{d}_{\omega}$ are different from the path derivatives $(\partial_t,\,\partial_{\omega})$ via the functional It\^o formulas (see \cite{buckdahn2015pathwise} and \cite[Section 2.3]{ekren2016viscosity-1}); if $u(\omega,t,x)$ is smooth enough w.r.t. $(\omega,t)$ in the path space,  for each $x\in\Lambda_T^0$,  we have the relation 
$$\mathfrak{d}_tu(\omega,t,x_{s,t-s})=\left(\partial_t+\frac{1}{2}\partial^2_{\omega\omega}\right)u(\omega,t,x_{s,t-s}),\quad \mathfrak{d}_{\omega}u(\omega,t,x_{s,t-s})=  \partial_{\omega}u(\omega,t,x_{s,t-s}),$$ for $0\leq s<t<T$,   which may be seen from   \cite[Section 6]{ekren2016viscosity-1} and \cite{buckdahn2015pathwise}.

\begin{rmk}\label{rmk-def-(iii)}
In condition (iii) in Definition \ref{defn-testfunc}, we endow the elements of $\mathscr C_{\sF}^1$ with a certain kind of regularity in a piecewise way, which would allow us to utilize the approximations (in $\mathscr C_{\sF}^1$) that are just piecewisely sufficiently regular in the proof for the uniqueness of viscosity solution in Section \ref{subsection-uniqueness}. In fact, the test function space $\mathscr C_{\sF}^1$ could be further enlarged. For instance, we may replace the H\"older continuity \eqref{holder-regulr} (that holds locally in time $t$) with a local one, i.e.,  let \eqref{holder-regulr} hold when $\left\|x_t-y_t\right\|_0$ is small enough. Nevertheless, in this work  we would not seek such a generality as it is sufficient to have the present version of  $\mathscr C_{\sF}^1$.
\end{rmk}

For each stopping time $t\leq T$, let $\mathcal{T}^t$ be the set of stopping times $\tau$ valued in $[t,T]$ and $\mathcal{T}^t_+$ the subset of $\mathcal{T}^t$ such that $\tau>t$ for each $\tau\in \mathcal{T}^t_+$. For each $\tau\in\mathcal T^0$ and $\Omega_{\tau}\in\sF_{\tau}$, we denote by $L^0(\Omega_{\tau},\sF_{\tau};\Lambda^0_{T})$ the set of $\Lambda^0_{T}$-valued $\sF_{\tau}$-measurable functions and define the restricted space
$$L^0(\Omega_{\tau},\sF_{\tau};\Lambda^0_{\tau})=\{  \xi(\tau\wedge \cdot): \xi\in L^0(\Omega_{\tau},\sF_{\tau};\Lambda^0_{T}) \}.$$ 
Here, we recall that for each $t\in[0,T]$ we think of $\Lambda^0_{t}$ as a complete subspace of $\Lambda^0_T$ as we do not distinguish $\eta\in\Lambda^0_{t}$ from $\eta(\tau\wedge \cdot) \in \Lambda^0_{T}$.
 
For each $k\in\bN^+$, $0\leq t\leq s\leq T$ and $\xi\in \Lambda_t$, define
\begin{align*}
\Lambda^{0,k;\xi}_{t,s}=\bigg\{x\in \Lambda_s:\,x(\tau)=\xi(\tau\wedge t)+ \int_{t\wedge \tau }^{\tau}g(r)\,dr, \,\,\tau\in[0,s], \quad &\text{for some }g\in L^{\infty}([0,T]) \\
 &\quad \text{with }\|g\|_{L^{\infty}([0,T])}\leq k  \bigg\},
\end{align*}
and in particular, we set $\Lambda^{0,k}_{0,t}=\cup_{\xi\in\bR^d}\Lambda^{0,k;\xi}_{0,t}$ for each $t\in [0,T]$. Then for each $\xi\in \Lambda_t^0$,  Arzel$\grave{\text{a}}$-Ascoli theorem indicates that each $\Lambda^{0,k;\xi}_{t,s}$ is compact in $\Lambda_s^0$. Moreover, it is obvious that $\cup_{k\in\bN^+} \Lambda^{0,k}_{0,s}$ is dense in $\Lambda_s^0$. In addition, by saying $(s,x)\rightarrow (t^+,\xi)$ for some $(t,\xi)\in [0,T)\times \Lambda_t^0$ we mean $(s,x)\rightarrow (t^+,\xi)$  with $s\in (t,T]$ and $x\in \cup_{k\in\bN^+} \Lambda_{t,s}^{0,k;\xi}$.

We now introduce the notion of viscosity solutions. For each $(u,\tau)\in \cS^{2}(\Lambda;\bR)\times \mathcal T^0$, $\Omega_{\tau}\in\sF_{\tau}$ with $\mathbb P(\Omega_{\tau})>0$ and $\xi\in L^0(\Omega_{\tau},\sF_{\tau};\Lambda^0_{\tau})$, we define for each $k\in \bN^+$,
{\small
\begin{align*}
\underline{\mathcal{G}}u(\tau,\xi;\Omega_{\tau},k):=\bigg\{
\phi\in\mathscr C^1_{\sF}:&
  \text{ there exists } \hat\tau \in  \mathcal T^{\tau}_+\text{ such that}\\
&(\phi-u)(\tau,\xi)1_{\Omega_{\tau}}=0=\essinf_{\bar\tau\in\mathcal T^{\tau}} E_{\sF_{\tau}}\left[\inf_{y\in \Lambda^{0,k;\xi}_{\tau,\bar\tau\wedge\hat\tau}}
(\phi-u)(\bar\tau\wedge \hat{\tau},y)
\right]1_{\Omega_{\tau}}  \text{ a.s.}
\bigg\},\\
\overline{\mathcal{G}}u(\tau,\xi;\Omega_{\tau},k):=\bigg\{
\phi\in\mathscr C^1_{\sF}:
&
  \text{ there exists } \hat\tau \in  \mathcal T^{\tau}_+\text{ such that}\\
&(\phi-u)(\tau,\xi)1_{\Omega_{\tau}}=0=\esssup_{\bar\tau\in\mathcal T^{\tau}} E_{\sF_{\tau}}\left[\sup_{y\in \Lambda^{0,k;\xi}_{\tau,\bar\tau\wedge\hat\tau} }
(\phi-u)(\bar\tau\wedge \hat{\tau},y)
\right]1_{\Omega_{\tau}}  \text{ a.s.}
%&\\
%\forall \, k\in \bN^+, \text{ for some } \hat\tau \in  \mathcal T^{\tau}_+\,&
\bigg\}.
\end{align*}
}
%It is obvious that if $\underline{\mathcal{G}}u(\tau,\xi;\Omega_{\tau},k)$ or $\overline{\mathcal{G}}u(\tau,\xi;\Omega_{\tau},k)$ is nonempty, we must have $0\leq\tau <T$ on $\Omega_{\tau}$. 

\begin{rmk}
In the above definitions of sub/superjets $\underline{\mathcal{G}}u(\tau,\xi;\Omega_{\tau},k)$ and $\overline{\mathcal{G}}u(\tau,\xi;\Omega_{\tau},k)$, due to (possibly non-Markovian type) randomness that is not addressed in Lukoyanov \cite{lukoyanov2007viscosity}, we adopt optimal stopping times after taking maximum/minimum over the compact subspaces of paths, which is, however, obvisouly different from taking optimal times  under certain classes of nonlinear expectations in \cite{ekren2014viscosity,ekren2016viscosity-1,ren2014overview}. As for the natural connections between optimal stopping times and non-Markovian type optimal controls, we refer to  \cite{Qiu2014weak,qiu2017viscosity,ren2014overview} for more discussions.
\end{rmk}

The definition of viscosity solutions then comes as follows.

\begin{defn}\label{defn-viscosity}
We say $u\in \cS^2 (\Lambda^0;\bR)$ is a viscosity subsolution (resp. supersolution) of SPHJ equation \eqref{SHJB}, if $u(T,x)\leq (\text{ resp. }\geq) G(x)$ for all $x\in\Lambda^0_T$ a.s., and for any $K_0\in \bN^+$, there exists $k\in\bN^+$ with $k\geq K_0$ such that for any $\tau\in  \mathcal T^0$, $\Omega_{\tau}\in\sF_{\tau}$ with $\mathbb P(\Omega_{\tau})>0$ and $\xi\in L^0(\Omega_{\tau},\sF_{\tau};\Lambda^0_{\tau})$ and any $\phi\in \underline{\cG}u(\tau,\xi;\Omega_{\tau},k)$ (resp. $\phi\in \overline{\cG}u(\tau,\xi;\Omega_{\tau},k)$), there holds
\begin{align}
&\text{ess}\liminf_{(s,x)\rightarrow (\tau^+,\xi)}
	  \left\{ -\mathfrak{d}_{s}\phi(s,x)-\bH(s,x,\nabla \phi(s,x)) \right\}  \leq\ \,0, \text{ for almost all } \omega\in\Omega_{\tau}
\label{defn-vis-sub}
\end{align}
\begin{align}
\text{(resp.} \quad &\text{ess}\!\!\limsup_{(s,x)\rightarrow (\tau^+,\xi)} 
	\!\! 
		\left\{ -\mathfrak{d}_{s}\phi(s,x)-\bH(s,x,\nabla \phi(s,x)) \right\}  \geq\ \,0,  \text{ for almost all } \omega\in\Omega_{\tau}\text{).}
\label{defn-vis-sup}
\end{align}
The function $u$ is a viscosity solution of SPHJ equation \eqref{SHJB} if it is both a viscosity subsolution and a viscosity supersolution of \eqref{SHJB}.
\end{defn}
\begin{rmk}\label{rmk-defn}
In the above definition, one may see that each viscosity subsolution (resp. supersolution) of SPHJ equation \eqref{SHJB} is associated to an infinite sequence of integers $1\leq \underline k_1<\underline k_2<\dots<\underline k_n<\dots$ (resp. $1\leq \overline k_1<\overline k_2<\dots<\overline k_n<\dots$) such that the required properties in Definition \ref{defn-viscosity} are holding for all the test functions in  $\underline{\cG}u(\tau,\xi;\Omega_{\tau},\underline k_i)$ (resp. $\overline{\cG}u(\tau,\xi;\Omega_{\tau},\overline k_i)$) for all $i\in\bN^+$. 
\end{rmk}

Throughout this paper, we define for each $\phi\in\mathscr C^1_{\sF}$, $v\in U$, $t\in[0,T]$, and $x_t\in \Lambda_t$,
\begin{align*}
\mathscr L^{v}\phi(t,x_t)=
 \mathfrak{d}_t \phi (t,x_t)      +\beta'(t,x_t,v)\nabla\phi(t,x_t).
\end{align*}

\begin{rmk}\label{rmk-bH}
In view of the assumption $ (\cA 1)$, for each $\phi\in \mathscr C_{\sF}^1$,  there exists an $\sF_t$-adapted process $\zeta^{\phi}\in L^2(\Omega\times [0,T])$ such that for a.e. $(\omega,t)\in \Omega\times [0,T]$, and all $x_t \in \Lambda_t$, we have
\begin{align*}
 &\Big|  -\mathfrak{d}_{t}\phi(t,x_t)-\bH(t,x_t,\nabla \phi(t,x_t)) 
  \Big| 
   \leq 
  \sup_{v\in U}
  \Big|  \mathscr L^{v}\phi(t,x_t) + f(t,x_t,v) \Big|
  \leq \zeta^{\phi}_t;
\end{align*}
meanwhile,  there exists a finite partition $0=\underline t_0<\underline t_1<\ldots<\underline t_n=T$,  such that  for any $0<\delta<\min_{0\leq j \leq n-1} |\underline t_{j+1}-\underline t_j|$,  there exists $L^{\phi}_{\alpha} \in (0,\infty) $ satisfying a.s. for almost all $t\in \cup_{0\leq j \leq n-1} (\underline t_j, \underline t_{j+1}-\delta]$ and all $x_t,\bar x_t\in \Lambda_t$, 
\begin{align}
 & \Big| 
  \left\{ -\mathfrak{d}_{t}\phi(t,x_t)-\bH(t,x_t,\nabla \phi(t,x_t))   \right\}
  -\left\{ -\mathfrak{d}_{t}\phi(t,\bar x_t)-\bH(t,\bar x_t,\nabla \phi(t,\bar x_t))  \right\}\Big|\nonumber\\
  &
  \leq 
  \sup_{v\in U}
  \Big| 
   \left( \mathscr L^{v}\phi(t,x_t) +f(t,x_t,v) \right)
  -\left( \mathscr L^{v}\phi(t,\bar x_t)  + f(t,\bar x_t,v)  \right)\Big|
  \nonumber
  \\
  &
  \leq L^{\phi}_{\alpha}  \left(\|x_t-\bar x_t\|^{\alpha}_0 +\|x_t-\bar x_t\| _0  \right),
 \label{R-Lip-const}
\end{align}
where $\alpha$ is the exponent associated to $\phi \in \mathscr C^1_{\sF}$.
%there exists some $\zeta^{\phi}\in \cL^4(0,T;
%\bR)$ such that for all $x,\bar x \in \bR^d, y\in\bR, z\in\bR^m, $ and all $(\theta,\gamma)\in\Theta\times \Gamma$,
%\begin{align}
%|F(s,x,y,z,\theta_s,\gamma_s)-F(s,\bar x,y,z,\theta_s,\gamma_s)|\leq \zeta^{\phi}_t |x-\bar x|, \quad\text{for a.e. }(
%\omega,t)\in \Omega\times [0,T]. \label{R-Lip-const}
%\end{align}
Therefore, the  essential limits in \eqref{defn-vis-sub} and \eqref{defn-vis-sup} are well-defined.
\end{rmk}

In Definition \ref{defn-viscosity}, the defined viscosity solution integrates the following three main aspects:

First, due to the path-dependence of the coefficients, the solution $u$ is path-wisely defined on the path space $C([0,T];\bR^d)$, and the lack of local compactness of the path space prompts us to define the random test functions (tangent from above or from below) in certain compact subspaces via optimal stopping times. In view of  the spaces $\underline{\cG}u(\tau,\xi;\Omega_{\tau},k)$ and $\overline{\cG}u(\tau,\xi;\Omega_{\tau},k)$, one may, however, see that we avoid the usage of the capacity and nonlinear expectation techniques for deterministic path-dependent PDEs (see \cite{ekren2014viscosity,ekren2016viscosity-1} for instance).

Second,  as the involved coefficients are  just measurable w.r.t. $\omega$ on the sample space $(\Omega,\sF)$ without any specified topology,  it is not appropriate to define the viscosity solutions in a pointwise manner w.r.t. $\omega\in (\Omega,\sF)$. This together with the  nonanticipativity constraint enlightens us to use relatively regular random fields in $\mathscr C_{\sF}^1$ as test functions; at each point $(\tau,\xi)$, the classes of test functions are also parameterized by measurable set $\Omega_{\tau}\in\sF_{\tau}$ and the type of compact subspaces of $C([0,T];\bR^d)$. 
 
Third, the test function space $ \mathscr C_{\sF}^1$ is expected to include the classical solutions. However, it is typical that the classical solutions may not be  differentiable in the time variable $t$ and $(\mathfrak{d}_tu,\mathfrak{d}_{\omega}u)$ may not be time-continuous but just measurable in $t$, which is also reflected in Definition \ref{defn-testfunc}; see \cite{DuQiuTang10,Tang-Wei-2013} for the \text{state}-dependent BSPDEs, or one may even refer to the standard theory of BSDEs that may be though of as trivial path-independent cases.  This nature leads to the usage of essential limits in \eqref{defn-vis-sub} and \eqref{defn-vis-sup}.

%%%%%%%%%%%%%%%%%%%%%%%%%%%%%%%%%%%%%%%%%%%%%%%%%%%%%%%%%%%%%%%%%%%%%%%%%%%%%%%%%%%%%%%%%%%%%%%%%%%%%%%%%%%%%%%%%%%%%%%%%%%%%%%%%%%%%%%%%%%%%%%%%%%%%%%%%%%%%%%%%%%%%%%%%%%%%%%%%%%%%%%%%%%%%%%%%%%%%%%%%%%%%%%%%%%%%%%%%%%%%%%%%%%%%%%%%%%%%%%%
\section{Existence of the viscosity solution}
\subsection{Some auxiliary results}
Under assumption $(\cA1)$ with the vanishing diffusion coefficients in stochastic (ordinary) differential equation \eqref{state-proces-contrl}, the following assertions are standard; see \cite{karatzas1998brownian,yong-zhou1999stochastic} for instance.
\begin{lem}\label{lem-SDE}
Let $(\cA1)$ hold. Given $\theta\in\cU$, for the strong solution of (stochastic) ODE \eqref{state-proces-contrl}, there exists $K>0$  such that, for any $0\leq r \leq t\leq s \leq T$,  and $\xi\in L^0(\Omega,\sF_r;\Lambda_r)$,
 %with $p\in[1,\infty)$,
 \\[3pt]
(i)   the two processes $\left(X_s^{r,\xi;\theta}\right)_{t\leq s \leq T}$ and $\left(X^{t,X_t^{r,\xi;\theta};\theta}_s\right)_{t\leq s\leq T}$ are indistinguishable;\\[2pt]
(ii)  $\max_{r\leq l \leq T} \left\|X^{r,\xi;\theta}_l\right\|_0 \leq K \left(1+ \|\xi\|_0\right)$ a.s.;\\[2pt]
(iii) $   d_0( X^{r,\xi;\theta}_s,\,X^{r,\xi;\theta}_t  )  \leq K \left(|s-t|+  |s-t|^{1/2} \right)$ a.s.;\\[2pt]
(iv) given another $\hat{\xi}\in L^0(\Omega,\sF_r;\Lambda_r)$, 
$$
\max_{r\leq l \leq T} \left\|X^{r,\xi;\theta}_l-X^{r,\hat\xi;\theta}_l\right\|_0 \leq K  \|\xi-\hat\xi\|_0
\quad \text{a.s.};
$$
(v) the constant $K$ depends only on $L,$ and $T$.
\end{lem}

%For the dynamic cost functional defined in \eqref{eq-cost-funct}, the following lemma is an immediate application of \cite[Theorem 4.7 and Lemma 6.5]{Peng-DPP-1997}.
%
%\begin{lem}\label{lem-J}
%Let $\theta\in\cU$. For any $0\leq t\leq T$ and any $\xi\in L^p(\Omega,\sF_t;\bR^d)$ with $p\in[1,\infty]$, we have 
%$$
%J(t,\xi;\theta)=E_{\sF_t}\left[\int_t^T\!\! f(s,X^{t,\xi;\theta}_s,\theta_s)\,ds +G(X^{t,\xi;\theta}_T) \right] \quad \text{ a.s.,}
%$$
%and 
%$$
%V(t,\xi)=\essinf_{v\in\cU} J(t,\xi;v),\quad \text{ a.s.}
%$$
%\end{lem}

The following regular properties of the value function $V$ hold in a similar way to \cite[Proposition 3.3]{qiu2017viscosity}  and the proof is omitted.
\begin{prop}\label{prop-value-func}
Let $(\cA 1)$ hold.
\\[4pt]
(i) For each  $t\in[0,T]$, $\eps\in (0,\infty)$, and $\xi\in L^0(\Omega,\sF_t;\Lambda_t)$, there exists $\bar{\theta}\in\cU$ such that
$$
E\left[ J(t,\xi;\bar{\theta})-V(t,\xi)\right]<\eps.
$$
(ii) For each $({\theta},x_0)\in\cU\times \bR^d$, $\left\{J(t,X_t^{0,x_0;\theta};{\theta})-V(t,X_t^{0,x_0;\theta})\right\}_{t\in[0,T]}$ is a supermartingale, i.e.,  for any $0\leq t\leq \tilde{t}\leq T$,
\begin{align}
V(t,X_t^{0,x_0;\theta})
\leq E_{\sF_t}V(\tilde{t},X_{\tilde{t}}^{0,x_0;{\theta}}) + E_{\sF_t}\int_t^{\tilde{t}}f(s,X_s^{0,x_0;{\theta}},\theta(s))\,ds,\,\,\,\text{a.s.}\label{eq-vfunc-supM}
\end{align}
(iii) For each $({\theta},x_0)\in\cU\times \bR^d$, $\left\{V(s,X_s^{0,x_0;{\theta}})\right\}_{s\in[0,T]}$ is a continuous process.
\\[3pt]
(iv) 
%For each $\theta\in\cU$,
%\begin{align}
%\int_{\bR^d}E\sup_{t\in[0,T]}|V(t,x+X_t^{\theta})|^2\,dx<\infty,
%%\quad (p,q)\in(1,\infty)\times[1,\infty),
%\label{eq-iv-lemm-value-funct}
%\end{align}
%and
 There exists $L_V>0$ such that for each $(\theta,t)\in\cU \times[0,T]$,
$$
|V(t,x_t)-V(t,y_t)|+|J(t,x_t;\theta)-J(t,y_t;\theta)|\leq L_V\|x_t-y_t\|_0,\,\,\,\text{a.s.},\quad \forall\,x_t,y_t\in\Lambda_t,
$$
with $L_V$ depending only on $T$ and $L$.
\\[3pt]
(v) With probability 1, $V(t,x)$ and $J(t,x;\theta)$ for each $\theta\in\cU$ are continuous  on $[0,T]\times\Lambda$ and 
$$ \sup_{(t,x)\in[0,T]\times\Lambda}   \max\left\{|V(t,x_t)|,\,|J(t,x_t;\theta)| \right\}       \leq L(T+1) \quad\text{a.s.}$$
\end{prop}

\begin{thm}\label{thm-DPP}
Let assumption $(\cA1)$ hold. For any stopping times $\tau,\hat\tau $ with $\tau\leq \hat\tau \leq T$, and any $ \xi\in L^0(\Omega,\sF_{\tau};\Lambda^0_{\tau})$,
% for some $p\in[1,\infty]$, 
 we have
\begin{align*}
V(\tau,\xi)=\essinf_{\theta\in\cU} E_{\sF_{\tau}}
\left[
\int_{\tau}^{\hat\tau} f\left(s,X_s^{\tau,\xi;\theta},\theta(s)\right)\,ds + V\left(\hat\tau,X^{\tau,\xi;\theta}_{\hat \tau}\right)
\right] \quad a.s.
\end{align*}
\end{thm}

The proof is similar to \cite[Theorem 3.4]{qiu2017viscosity}, but with some careful compactness arguments about path (sub)spaces; therefore, the proof is postponed to the appendix.

\subsection{Existence of the viscosity solution}
We first generalize an It\^o-Kunita formula by Kunita \cite[Pages 118-119]{kunita1981some} for the composition of random fields and stochastic differential equations to our \textit{path-dependent} setting. Recall that for each $\phi\in\mathscr C^1_{\sF}$, $v\in U$, $t\in[0,T]$, and $x_t\in \Lambda_t$,
\begin{align*}
\mathscr L^{v}\phi(t,x_t)=
 \mathfrak{d}_t \phi (t,x_t)      +\beta'(t,x_t,v)\nabla\phi(t,x_t).
\end{align*}

\begin{lem}\label{lem-ito-wentzell}
  Let assumption $(\cA1)$ hold.
 Suppose  $u\in\mathscr C_{\sF}^1$ with the associated partition $0=\underline t_0<\underline t_1<\ldots<\underline t_n=T$. Then, for each $\theta\in\cU$, it holds almost surely that, for each $\underline t_j \leq \varrho \leq \tau < \underline t_{j+1}$, $j=0,\ldots,n-1$, and $x_{\varrho}\in \Lambda_{\varrho}$, it holds that
    \begin{align}
  u(\tau,X^{{\varrho},x_{\varrho};\theta}_{\tau}) -u({\varrho},x_{\varrho})
= 
     	\!\int_{\varrho}^{\tau}  
	 \mathscr L^{\theta(s)} u\left(s,X^{{\varrho},x_{\varrho};\theta}_s \right)    \,ds      
	 +\int_{\varrho}^{\tau} \mathfrak{d}_{\omega}u(r,X^{{\varrho},x_{\varrho};\theta}_r)    
     \,dW(r),\text{ a.s.} \label{eq-ito}
   \end{align}
\end{lem}

\begin{proof}
%Due to the time continuity of both sides of relation \eqref{eq-ito}, we only need to verify that equality \eqref{eq-ito} holds a.s. for each $t\in[0,T]$; w.l.o.g., we only prove \eqref{eq-ito} when $t=T$.

In view of the time continuity, w.l.o.g., we only prove \eqref{eq-ito} for $ {\tau} \in (0,\underline t_1)$, ${\varrho}=0$ and $x_0=x\in \bR^d$. For each $N\in\bN^+$ with $N>2$, let $t_i=\frac{i{\tau}}{N}$ for $i=0,1,\dots,N$. Then, we get a partition of $[0,{\tau}]$ with $0=t_0<t_1<\cdots<t_{N-1}<t_N={\tau}$. For each $\theta\in\cU$, let
$$
^N\!X(t)=\sum_{i=0}^{N-1} X^{0,x;\theta}({t_i})1_{[t_i,t_{i+1})(t)} + X^{0,x;\theta}({\tau})1_{\{{\tau}\}}(t), \quad \text{for }t\in[0,{\tau}],
$$
and $^N\!X_{t-}(s)= \, ^N\!X(s)1_{[0,t)}(s) + \lim_{r\rightarrow t^-} \, ^N\!X(r) 1_{\{t\}}(s)$, for $0\leq s\leq t\leq {\tau}$. In view of the path continuity, we have the following approximation:
$$
\lim_{N\rightarrow \infty} \|X^{0,x;\theta}_t-^N\!X_t\|_0+\|X^{0,x;\theta}_t-^N\!X_{t-}\|_0=0\quad \text{for all }t\in(0,{\tau}],\quad \text{a.s.}
$$

For each $i\in\{0,\cdots,N-1\}$, 
\begin{align*}
u(t_{i+1},^N\!X_{t_{i+1}})-u(t_i,^N\!X_{t_i})
& = 
	u(t_{i+1},^N\!X_{t_{i+1}-})-u(t_i,^N\!X_{t_i})
	+u(t_{i+1},^N\!X_{t_{i+1}})-u(t_{i+1},^N\!X_{t_{i+1}-})
		\\
&:= I^i_1+I^i_2,
\end{align*}
where by $u\in\mathscr C_{\sF}^1$,
\begin{align*}
I^i_1=u(t_{i+1},^N\!X_{t_{i+1}-})-u(t_i,^N\!X_{t_i})
=\int_{t_i}^{t_{i+1}}  \mathfrak{d}_tu(s,^N\!X_{s-})\,ds 
	+ \int_{t_i}^{t_{i+1}} \mathfrak{d}_{\omega}u(r,^N\!X_{r-})    
     \,dW(r),
\end{align*}
and by the definition of vertical derivative and the integration by parts formula
\begin{align*}
I^i_2
&=u(t_{i+1},^N\!X_{t_{i+1}})-u(t_{i+1},^N\!X_{t_{i+1}-})\\
&=
	u\left( t_{i+1},^N\!X_{t_{i+1}-}^{^N\!X(t_{i+1})-^N\!X(t_{i})} \right)-u(t_{i+1},^N\!X_{t_{i+1}-})
		\\
&=
	  \int_{t_i}^{t_{i+1}} (\nabla u)'(t_{i+1},^N\!X_{t_{i+1}-} ^{X(s)-X(t_{i}) }) \beta(s,X^{0,x;\theta}_s,\theta(s))\,ds.
\end{align*}
Recall that
\begin{align}
u({\tau},^N\!X_{\tau})-u(0,x)
%=\sum_{i=0^{N-1}} u(t_{i+1},^N\!X_{t_{i+1}})-u(t_i,^N\!X_{t_i})
=\sum_{i=0}^{N-1}I^i_1+I^i_2. \label{eq-docomp}
\end{align}
Then, the dominated convergence and the dominated convergence theorem for stochastic integrals (\cite[Chapter IV, Theorem 32]{protter2005stochastic})  imply that the Lebesgue integrals converge almost surely and the stochastic integrals in probability, with the limits being the corresponding terms in \eqref{eq-ito}. Finally, we get 
    \begin{align*}
  u({\tau},X^{0,x;\theta}_{\tau}) -u(0,x)
= 
     	\!\int_0^{\tau}    \mathscr L^{\theta(s)} u\left(s,X^{0,x;\theta}_s \right)    \,ds      +\int_0^{\tau} \mathfrak{d}_{\omega}u(r,X^{0,x;\theta}_r)    
     \,dW(r),\quad \text{a.s.}
   \end{align*}
\end{proof}

%%%%%%%%%%%%%%%%%%%%%%%%%%%%%%%%%%%%%%%%%%%%%%%%%%%%%%%%%%%%
%%%%%%%%%%%%%%%%%%%%%%%%%%%%%%%%%%%%%%%%%%%%%%%%%%%%%%%%%%%%
%%%%%%%%%%%%%%%%%%%%%%%%%%%%%%%%%%%%%%%%%%%%%%%%%%%%%%%%%%%%
%%%%%%%%%%%%%%%%%%%%%%%%%%%%%%%%%%%%%%%%%%%%%%%%%%%%%%%%%%%%

Following is the existence of the viscosity solution.

\begin{thm}\label{thm-existence}
 Let $(\cA1)$ hold. The value function $V$ defined by \eqref{eq-value-func} is a viscosity solution of the stochastic path-dependent Hamilton-Jacobi equation \eqref{SHJB}.
\end{thm}

\begin{proof}
\textbf{Step 1.} 
In view of Proposition \ref{prop-value-func}, we have $V\in\cS^{\infty}(\Lambda;\bR)$. To the contrary, suppose that for any $k\in \bN^+$ with $k\geq K_0$ for some $K_0\in\bN^+$,  there exists $\phi\in \underline\cG V(\tau,\xi_{\tau};\Omega_{\tau},k)$ with $\tau\in \mathcal T^0$, $\Omega_{\tau}\in\sF_{\tau}$, $\mathbb P(\Omega_{\tau})>0$, and $\xi_{\tau}\in L^0(\Omega_{\tau},\sF_{\tau};\Lambda_{\tau}^0)$,  such that   there exist $\eps>0$ and $\Omega'\in\sF_{\tau}$ with $\Omega'\subset \Omega_{\tau}$, $\bP(\Omega')>0$, satisfying a.e. on $\Omega'$,
{\small
\begin{align}
\lim_{\tilde\delta\rightarrow 0^+}\essinf_{
s \in (\tau,(\tau+\tilde\delta^2 )\wedge T],\,x\in B_{\tilde\delta}(\xi_{\tau}) \cap \Lambda^{0,k;\xi_{\tau}}_{\tau,s\wedge T}	
			} 
		\{ -\mathfrak{d}_{s}\phi(s,x)-\bH(s,x,\nabla\phi(s,x) )\}  \geq 2\,\eps.
		\label{eq-ex-1}
\end{align}
}

Notice that, associated to $\phi\in\mathscr C_{\sF}^1$, there is a partition $0=\underline t_0<\underline t_1<\ldots<\underline t_n=T$. W.l.o.g., we assume that there exists $\tilde{\delta}\in (0,1)$ with $2 \tilde \delta^2 <\min_{0\leq j\leq n-1} (\underline t_{j+1}-\underline t_j)$ such that $\Omega'=\{[\tau,\tau+2\tilde{\delta}^2] \subset [\underline t_j,\underline t_{j+1})\}$ for some $j\in\{0,\ldots, n-1\}$.

Choose the positive integer $k>L$ and let $\hat{\tau}$ be the stopping time associated to $\phi\in \underline\cG V(\tau,\xi_{\tau};\Omega_{\tau},k)$. Note that we may think of $\xi$ valued in $\Lambda^0_T$, with $\xi(t)=\xi_{t\wedge \tau}(t)$ for all $t\in [0,T]$. By assumption (ii) of $(\cA 1)$ and the measurable selection theorem, there exists $\bar\theta\in \cU$ such that for almost all $\omega\in\Omega'$,
$$
- \mathscr L^{\bar\theta(s)}\phi(s,\xi_s)  -f(s,\xi_s,\bar\theta(s))   \geq   -\mathfrak{d}_{s}\phi(s,\xi_s)-\bH(s,\xi_s,\nabla\phi(s,\xi_s)) -       \eps,
$$
for almost all $s$ satisfying $\tau\leq s< (\tau+\tilde{\delta}^2)\wedge T$. This together with \eqref{eq-ex-1} implies that
\begin{align*}
\lim_{\tilde{\delta}\rightarrow 0^+}\essinf_{\tau < s<\left(\tau+\tilde\delta^2\right)\wedge T} 
		\{ - \mathscr L^{\bar\theta(s)}\phi(s,\xi_s)  -f(s,\xi_s,\bar\theta(s)) \}  \geq \eps,\quad \text{a.e. on }\Omega'.
\end{align*}

In view of Remark \ref{rmk-bH}, Lemma \ref{lem-SDE}, the dynamic programming principle of Theorem \ref{thm-DPP}, and the generalized  It\^o-Kunita formula of Lemma \ref{lem-ito-wentzell}, we have for almost all $\omega\in\Omega'$, 
{\small
\begin{align*}
0 
&\geq \liminf_{h\rightarrow 0^+} \frac{1}{h}
	E_{\sF_{\tau}}\left[ 
		  (\phi-V)(\tau,\xi_{\tau}) -(\phi-V)\left((\tau+h)		\wedge\hat\tau,X_{(\tau+h)		\wedge\hat\tau}^{\tau,\xi_{\tau};\bar\theta}\right)   \right]
\\
&
\geq \liminf_{h\rightarrow 0^+} \frac{1}{h } 
	E_{\sF_{\tau}} \left[ 
		   \phi(\tau,\xi_{\tau}) - \phi\left((\tau+h)\wedge 		\hat\tau,X_{(\tau+h)  		\wedge\hat\tau}^{\tau,\xi_{\tau};\bar\theta}\right) 
		   -\int_{\tau}^{(\tau+h)		\wedge\hat\tau} f(s,X_s^{\tau,\xi_{\tau};\bar\theta},\bar\theta(s))\,ds    \right]
\\
&
=\liminf_{h\rightarrow 0^+}\frac{1}{h} E_{\sF_{\tau}} \left[
	\int_{\tau}^{(\tau+h)			\wedge\hat\tau}
		   \left( -\mathscr L^{\bar\theta(s)}  \phi(s,X_s^{\tau,\xi_{\tau};\bar\theta})		- f(s,X_s^{\tau,\xi_{\tau};\bar\theta},\bar\theta(s))   \right)\,ds \right]
 \\
% &
% 	\geq
% 	\frac{1}{h} E_{\sF_{\tau}} \bigg[
%		\int_{\tau}^{(\tau+h)\wedge T}
%		   \left( -\mathscr L^{\bar\theta({s})}  \phi(s,X_{s	}^{\tau,\xi_{\tau};\bar\theta})		- f(s,X_{s	}^{\tau,\xi_{\tau};\bar\theta},\bar\theta({s}))   \right)\,ds
%\\
%&   
%	 \quad\quad\quad\quad
%		   -   1_{\{\tau+h>\hat\tau\}      }
%		   	\int_{\tau}^{T}
%		  	 \left|   -\mathscr L^{\bar\theta(s)}  \phi(s,X_{s	}^{\tau,\xi_{\tau};\bar\theta})		- f(s,X_{s	}^{\tau,\xi_{\tau};\bar\theta},\bar\theta({s}))   \right|\,ds
%			 \bigg]
%\\
&\geq
 	\liminf_{h\rightarrow 0^+} \frac{1}{h} E_{\sF_{\tau}} \bigg[
		\int_{\tau}^{(\tau+h)\wedge \hat \tau}
		   \left( -\mathscr L^{\bar\theta({s})}  \phi(s,\xi_s)		- f(s,\xi_s,\bar\theta({s}))   \right)\,ds
\\
&\quad
		-\int_{\tau}^{ (\tau+h)\wedge \hat \tau}
		   \left| -\mathscr L^{\bar\theta({s})}  \phi(s,\xi_s)		- f(s,\xi_s,\bar\theta({s}))  +  \mathscr L^{\bar\theta({s})}  \phi(s,X_{s	}^{\tau,\xi_{\tau};\bar\theta})
		   +  f(s,X_{s	}^{\tau,\xi_{\tau};\bar\theta},\bar\theta({s}))
		   \right|\,ds \bigg]
%\\
%&    \quad\quad\quad
%		   -    1_{\{\tau+h>\hat\tau\}		   }
%		   	\int_{\tau}^{T}
%		  	 \left|   -\mathscr L^{\bar\theta(s)}  \phi(s,X_{s	}^{\tau,\xi_{\tau};\bar\theta})		- f(s,X_{s	}^{\tau,\xi_{\tau};\bar\theta},\bar\theta({s}))   \right|\,ds
%			 \bigg]
\\
&\geq
	 \eps 	 
	- \limsup_{h\rightarrow 0^+}\frac{1}{h} E_{\sF_{\tau}} \bigg[ 
	\int_{\tau}^{ (\tau+h)\wedge \hat \tau} L^{\phi}_{\alpha} \left(  \| X_{s	}^{\tau,\xi_{\tau};\bar\theta} -\xi_s\|_0^{\alpha} + \| X_{s	}^{\tau,\xi_{\tau};\bar\theta} -\xi_s\|_0 \right) ds
	\bigg]\\
&\geq
	\eps 
	- \limsup_{h\rightarrow 0^+}E_{\sF_{\tau}} \left[  \frac{(\hat\tau \wedge(\tau+h))-\tau}{h}   \right]
	\cdot
	 L^{\phi}_{\alpha}\cdot E_{\sF_{\tau}} \bigg[ 
	   \left| d_0( X_{(\tau+h)\wedge \hat \tau	}^{\tau,\xi_{\tau};\bar\theta}, \xi_{  \tau}) \right|^{\alpha} + d_0( X_{(\tau+h)\wedge \hat \tau	}^{\tau,\xi_{\tau};\bar\theta}, \xi_{  \tau}) 
	\bigg]
	\\
&\geq \eps-
 \limsup_{h\rightarrow 0^+} E_{\sF_{\tau}}\left[\frac{(\hat\tau \wedge(\tau+h))-\tau}{h} \right] \cdot\left\{
L^{\phi}_{\alpha}\cdot K \left((h+\sqrt{h})^{\alpha} + h+\sqrt{h}\right)
\right\}\\
&=\eps>0,
\end{align*}
}
where the exponent $\alpha$ is associated to $\phi\in \mathscr C _{\sF}^1$, $L^{\phi}_{\alpha}$ from Remark \ref{rmk-bH} and the constant $K$ from Lemma \ref{lem-SDE}. This gives rise to a contradiction. Hence, $V$ is a viscosity subsolution.\medskip

%%%%%%%%%%%%%%%%%%%%%%%%%%%%%%%%%%%%%%%%%%%%%%%%%%%%%%%%%%%%%%%%%%%%%%%%%%%%%%%%%%%%%%
\textbf{Step 2.}
To prove that $V$ is a viscosity supersolution of \eqref{SHJB}, we argue with contradiction like in \textbf{Step 1}. To the contrary, assume that for any  $k\in \bN^+$ with $k\geq K_0$ for some $K_0\in\bN^+$, there exists $\phi\in \overline\cG V(\tau,\xi_{\tau};\Omega_{\tau},k)$ with $\tau\in \mathcal T^0$, $\Omega_{\tau}\in\sF_{\tau}$, $\mathbb P(\Omega_{\tau})>0$, and $\xi_{\tau}\in L^0(\Omega_{\tau},\sF_{\tau};\Lambda_{\tau}^0)$
such that 
there exist $\eps >0$ and $\Omega'\in\sF_{\tau}$ with $\Omega'\subset \Omega_{\tau}$, $\bP(\Omega')>0$, satisfying a.e. on $\Omega'$,
\begin{align*}
\lim_{h\rightarrow 0^+} \esssup_{
s \in (\tau,(\tau+4\tilde\delta^2) \wedge T],\,x\in B_{2\tilde\delta}(\xi_{\tau}) \cap \Lambda^{0,k;\xi_{\tau}}_{\tau,s\wedge T}	
			} 
		\{ -\mathfrak{d}_{s}\phi(s,x)-\bH(s,x,\nabla\phi(s,x) )\}  \leq -\eps.
\end{align*}

We take the interger $k>L$ and let  $\hat{\tau}$ be the stopping time corresponding to the fact $\phi\in \overline\cG V(\tau,\xi_{\tau};\Omega_{\tau},k)$. Again, we think of $\xi$ valued in $\Lambda^0_T$, with $\xi(t)=\xi_{t\wedge \tau} (t)$ for all $t\in [0,T]$.

Again, notice that, associated to $\phi\in\mathscr C_{\sF}^1$, there is a partition $0=\underline t_0<\underline t_1<\ldots<\underline t_n=T$. W.l.o.g., we assume that there exists $\tilde{\delta}\in (0,1)$ such that $2 \tilde \delta ^2 <\min_{0\leq j\leq n-1} (\underline t_{j+1}-\underline t_j)$ and $\Omega'=\{[\tau,\tau+2\tilde{\delta}^2] \subset [\underline t_j,\underline t_{j+1})\}$ for some $j\in\{0,\ldots, n-1\}$.

For each $\theta\in\cU$, define
$\tau^{\theta}=\inf\{ s>\tau: X^{\tau,\xi_{\tau};{\theta}}_s \notin B_{\tilde\delta}(\xi_{\tau}) \}$. 
Then $\tau^{\theta}  >\tau $ and moreover, setting $h = \frac{\tilde \delta ^2}{4}$ and using Chebyshev's inequality, we obtain
 \begin{align}
 E_{\sF_{\tau}}\left[1_{\{\tau^{\theta} <\tau+h\}}\right]
 &=
 	E_{\sF_{\tau}}\left[1_{\{ \max_{\tau\leq s\leq \tau +h} |X^{\tau,\xi_{\tau};\theta}(s) -\xi(\tau)| + \sqrt{h}> {\tilde\delta} \}}\right]  
\nonumber\\
&
 \leq 
 	\frac{1}{(\tilde\delta-\sqrt{h})^8}  E_{\sF_{\tau}} \left[ \max_{\tau\leq s\leq \tau +h} |X^{\tau,\xi_{\tau};\theta} (s) -\xi(\tau)|^8\right]
\nonumber\\
&
\leq 
	\frac{K^8}{(\tilde\delta-\sqrt{h})^8}   (h+\sqrt{h})^8  
	\nonumber \\
&
\leq
	\frac{256 \cdot K^8}{\tilde\delta ^8}   (h+\sqrt{h})^8  
	\nonumber\\
&\leq
\tilde{C} h^4
	\quad \text{ a.s..}
	\label{est-tau}
 \end{align}
Here, the constant $K$ is from Lemma \ref{lem-SDE} and independent of the control $\theta$, and thus, the constant $\tilde C $ is independent of $\theta$ as well.

In view of Remark \ref{rmk-bH}, Theorem \ref{thm-DPP}, Lemma \ref{lem-ito-wentzell} and estimate \eqref{est-tau}, we have for almost all $\omega\in\Omega'$, 
{\small
\begin{align*}
0
&= \lim_{h\rightarrow 0^+}
	\frac{V(\tau,\xi_{\tau})-\phi(\tau,\xi_{\tau})}{h}
\\
&= \lim_{h\rightarrow 0^+}
	\frac{1}{h} \essinf_{\theta\in\cU} E_{\sF_{\tau}}\left[
	\int_{\tau}^{\hat{\tau} \wedge (\tau+h)} 
		f(s,X_s^{\tau,\xi_{\tau};\theta},\theta(s))\,ds
		+V\left(\hat{\tau} \wedge (\tau+h), X_{\hat{\tau}\wedge (\tau+h)}^{\tau,\xi_{\tau};\theta}\right) -\phi(\tau,\xi_{\tau})
		\right]
\\
&\geq 
	\liminf_{h\rightarrow 0^+}
	\frac{1}{h} \essinf_{\theta\in\cU} E_{\sF_{\tau}}\left[
	\int_{\tau}^{\hat{\tau} \wedge (\tau+h)} 
		f(s,X_s^{\tau,\xi_{\tau};\theta},\theta(s))\,ds
		+\phi\left(\hat{\tau} \wedge (\tau+h), X_{\hat{\tau}\wedge (\tau+h)}^{\tau,\xi_{\tau};\theta}\right) -\phi(\tau,\xi_{\tau})
		\right]
\\
&= 
	\liminf_{h\rightarrow 0^+}
	\frac{1}{h} \essinf_{\theta\in\cU} E_{\sF_{\tau}}\left[
	\int_{\tau}^{\hat{\tau} \wedge (\tau+h)} 
		\left( \mathscr L^{\theta(s)}\phi\left(s,X_s^{\tau,\xi_{\tau};\theta}\right)+f(s,X_s^{\tau,\xi_{\tau};\theta},\theta(s))\right)\,ds
		\right]
\\
%&\geq
%	\frac{1}{h} \essinf_{\theta\in\cU} E_{\sF_{\tau}}\bigg[
%	\int_{\tau}^{\hat{\tau}\wedge \tau^{\theta} \wedge (\tau+h)} 
%		f(s,X_s^{\tau,\xi;\theta},\theta_s)\,ds
%		+\phi\left( \tau+h, X_{\tau+h}^{\tau,\xi;\theta}\right)1_{\{\tau+h\leq\tau^{\theta}\wedge \hat\tau\}}-\phi(\tau,\xi)
%\\
%&\quad\quad\quad\quad
%		+V\left(\hat{\tau}\wedge \tau^{\theta} \wedge (\tau+h), X_{\hat{\tau}\wedge \tau^{\theta} \wedge (\tau+h)}^{\tau,\xi;\theta}\right)1_{\{\hat\tau < \tau+h\}\cup\{\tau+h>		\tau^{\theta}\}}  
%		\bigg]
%\\
&\geq
	\liminf_{h\rightarrow 0^+}
	\frac{1}{h} \essinf_{\theta\in\cU} E_{\sF_{\tau}}\bigg[
		\int_{\tau}^{\tau^{\theta}\wedge (\tau+h)\wedge \hat\tau} \bigg(
		\mathscr L^{\theta(s)}\phi\left(s,X_s^{\tau,\xi_{\tau};\theta} \right)+f(s,X_s^{\tau,\xi_{\tau};\theta},\theta(s))
      		  \bigg) \,ds\\
&	\quad\quad\quad
		  -1_{\{\hat\tau >\tau^{\theta}\}\cap\{\tau+h>		\tau^{\theta}\}}   \int_{\tau}^{(\tau+h)\wedge \hat\tau} \left| 
	\mathscr L^{\theta(s)}\phi\left(s,X_s^{\tau,\xi_{\tau};\theta} \right)+f(s,X_s^{\tau,\xi_{\tau};\theta},\theta(s))
      		  \right| \,ds
		\bigg]
\\
&\geq
	\liminf_{h\rightarrow 0^+}
	\frac{1}{h} \essinf_{\theta\in\cU} E_{\sF_{\tau}}\bigg[
		\int_{\tau}^{(\tau+h)\wedge \hat\tau} \bigg(
		\mathscr L^{\theta(s)}\phi\left(s,X_{s\wedge\tau^{\theta}}^{\tau,\xi_{\tau};\theta} \right)+f(s,X_{s\wedge\tau^{\theta}}^{\tau,\xi_{\tau};\theta},\theta(s))
      		  \bigg) \,ds\\
& \quad
-  1_{\{\hat\tau>\tau^{\theta}\}\cap \{\tau+h>\tau^{\theta}\}   }
 \int_{\tau}^{(\tau+h)\wedge \hat\tau} \left| 
	\mathscr L^{\theta(s)}\phi\left(s,X_{s\wedge \tau^{\theta}}^{\tau,\xi_{\tau};\theta} \right)+f(s,X_{s\wedge \tau^{\theta}}^{\tau,\xi_{\tau};\theta},\theta(s))
      		  \right| \,ds
\\
&	\quad
		  - 1_{\{\hat\tau >\tau^{\theta}\}\cap\{\tau+h>\tau^{\theta}\}}  	   
		  \int_{\tau}^{(\tau+h)\wedge \hat\tau} \left| 
	\mathscr L^{\theta(s)}\phi\left(s,X_s^{\tau,\xi_{\tau};\theta} \right)+f(s,X_s^{\tau,\xi_{\tau};\theta},\theta(s))
      		  \right| \,ds
		\bigg]
\\
& \geq
	\eps 
	- \limsup_{h\rightarrow 0^+} \frac{1}{h} \esssup_{\theta\in\cU} \left( E_{\sF_{\tau}}\left[ 1_{\{\tau+h>\tau^{\theta}\}} 		\right]\right)^{1/2}
	\\
&	\quad\quad
		  \cdot \left( E_{\sF_{\tau}} \left| \int_{\tau}^{(\tau+h)\wedge \hat\tau} \left| 
	\mathscr L^{\theta(s)}\phi\left(s,X_{s\wedge \tau^{\theta}}^{\tau,\xi_{\tau};\theta} \right)+f(s,X_{s\wedge \tau^{\theta}}^{\tau,\xi_{\tau};\theta},\theta(s))
      		  \right| \,ds \right|^2
		\right)^{1/2}\\
&
	- \limsup_{h\rightarrow 0^+} 
		\frac{1}{h} \esssup_{\theta\in\cU} 		 \left( E_{\sF_{\tau}}\left[  
		1_{\{\tau+h>\tau^{\theta}\}}		\right]\right)^{1/2}
\\
&\quad\quad	\cdot	
		\left( E_{\sF_{\tau}} \left| \int_{\tau}^{(\tau+h)\wedge \hat\tau} \!\! \left| 
	\mathscr L^{\theta(s)}\phi\left(s,X_{s }^{\tau,\xi_{\tau};\theta} \right)+f(s,X_{s }^{\tau,\xi_{\tau};\theta},\theta(s))
      		  \right| \,ds \right|^2
		\right)^{1/2}
\\
& \geq 
\eps
- 2 \limsup_{h\rightarrow 0^+}   h^{3/2} \tilde C ^{1/2} \left( E_{\sF_{\tau}} \left[  \int_{\tau}^{(\tau+h)\wedge \hat\tau} \big| \zeta^{\phi}_s\big|^2\, ds  \right] \right)^{1/2}
\\
&= \eps>0,
\end{align*}
}
where $h=\frac{\tilde \delta^2}{4}$ and we have used the analysis in Remark \ref{rmk-bH} as well as the fact that for almost all $\omega\in\Omega'$, $\tau +h <\hat\tau$ when $h>0$ is small enough. Hence, a contradiction occurs and $V$ is a viscosity supersolution of SPHJ equation \eqref{SHJB}.  
\end{proof}

\section{Uniqueness}

Due to the  nonanticipativity constraint on the unknown function, the conventional variable-doubling techniques for deterministic Hamilton-Jacobi equations are not applicable to the viscosity solution to SPHJ equations like \eqref{SHJB}, basically because in our random setting, the extreme quadruples, as random functions of $\omega\in\Omega$,  attaining the extreme values of the penalized functional ``doubling the number of variables" fail to be $(\sF_s)_{s\geq 0}$-adapted; refer to \cite{ekren2016viscosity-1,lukoyanov2007viscosity,ren2014overview} for the applications of variable-doubling techniques to the \textit{deterministic} path-dependent PDEs. In this work, we  prove instead a comparison principle that is weak in the sense that the comparison relations are just holding in compact subspaces; the uniqueness is then derived on basis of this weak comparison principle. 
%which together with a modified version of Perron's method yields the uniqueness of viscosity solution through regular approximations;  
%we note that this strategy is inspired by those proposed in \cite{ekren2014viscosity,ekren2016viscosity-1} for the \textit{deterministic} path-dependent PDEs.

\subsection{A weak comparison principle}
\begin{prop}\label{prop-cor-cmp}
Let $(\cA1)$ hold and $u$ be a viscosity subsolution (resp. supersolution)  of SPHJ equation \eqref{SHJB}. Then there is an infinite sequence of integers $1\leq \underline k_1<\underline k_2<\cdots<\underline k_n<\cdots$ (resp., $1\leq \overline k_1<\overline k_2<\cdots<\overline k_n<\cdots$), such that for each $i\in\bN^+$, $\phi_i\in\mathscr C^1_{\sF}$ satisfying %$(u-\phi)^+ (\text{resp. }(\phi-u)^+)\in \cS^{2}(\Lambda^0;\bR)$, 
$\phi_i(T,x)\geq (\text{resp. }\leq) G(x)$ for all $x\in\Lambda_{T}^{0 }$
%$x\in\Lambda_{0,T}^{0,\underline k_i}$ (resp., $x\in\Lambda_{0,T}^{0,\overline k_i}$) 
a.s. and  
\begin{align*}
\text{ess}\liminf_{(s,x)\rightarrow (t^+,y)}
	  \left\{  -\mathfrak{d}_{s}\phi_i(s,x)-\bH(s,x,\nabla \phi_i(s,x) ) \right\} \geq 0, \text{ a.s.,}
\\
\text{(resp. }
\text{ess}\limsup_{(s,x)\rightarrow (t^+,y)}
	 \left\{  -\mathfrak{d}_{s}\phi_i(s,x)-\bH(s,x,\nabla \phi_i(s,x) ) \right\} \leq 0\text{, a.s.)}
\end{align*}
for  each $t\in[0,T)$ with $y\in\Lambda_{0,t}^{0,\underline k_i}$ (resp., $y\in\Lambda_{0,t}^{0,\overline k_i}$), it holds a.s. that $u(t,x)\leq$ (resp., $\geq$) $\phi_i(t,x)$,  for each $t\in[0,T]$ with $x\in\Lambda_{0,t}^{0,\underline k_i}$.
\end{prop}
\begin{proof}
%W.l.o.g, we may assume $u^+\in \cS^{\infty}(C_0(\bR^d))$; otherwise, we consider $u-\|u\|_{\cS^{\infty}(C(\bR^d))}$. 
We prove the case when $u$ is a viscosity supersolution, then the proof for the viscosity subsolution will be following similarly. First, by Definition \ref{defn-viscosity} and Remark \ref{rmk-defn}, the viscosity supersolution $u$ is associated to an infinite sequence of integers $1\leq \overline k_1<\overline k_2<\cdots<\overline k_n<\cdots$. Given $i\in\bN^+$, suppose that, to the contrary, there holds $u(t,\bar x_t)<\phi_i(t,\bar x_t)$ with a positive probability  at some point $(t,\bar x_t) $ with $t\in[0,T)$ and $\bar x_t\in \Lambda_{0,t}^{0,\overline k_i}$. 
Thus, we have  $\bar x_t\in \Lambda^{0,\overline k_i;\bar \xi_0}_{0,t}$ for some $\bar \xi_0 \in  \bR^d$. W.l.o.g., we assume $\bar \xi_0=0$. Thus, there exists $\delta>0$ such that $\bP( \overline\Omega_t)>0$ with $\overline\Omega_t:=\{\phi_i(t,\bar x_t)-u(t,\bar x_t) >\delta  \}$.  

%Let
% \begin{equation}\label{bump-func}
% \rho(x)=
% \begin{cases}
% \tilde{c} \,  e^{\frac{1}{|x|^2-1}}&\quad \text{if } |x| < 1;\\
% 0&\quad \text{otherwise};
% \end{cases}
% \quad \mbox{with}\quad \tilde{c}:=\left(     \int_{|x| < 1} e^{\frac{1}{x^2-1}}\,dx    \right)^{-1},
% \end{equation}
% %and we define mollifier $\rho_l(x)=l^d\rho(lx)$, $x\in\bR^d$ for each $l\in\bN^+$.  
% and set
%\begin{align*}
%h(x)=\int_{\bR^d} 1_{\{|y|>1 \}} \left(  |y|-1\right) \rho(x-y)\,dy,\quad x\in\bR^d. 
%\end{align*}
%Then the function $h(x)$ is convex and continuously differentiable with $h({0})=0$, $h(x)>0$ whenever $|x|>0$, and 
%\begin{align}
%h(x)> |x|-2,\quad    |Dh(x)|\leq 1 \text{ for any } x\in\bR^d.\label{h-linear-growth}
%\end{align}

By the compactness of $\Lambda^{0,{\overline k_i};0}_{0,t}$ in $\Lambda_t^0$ and the measurable selection theorem, there exists $\xi^{\overline k_i}_t\in L^0(\overline\Omega_t,\sF_t;\Lambda^{0,{\overline k_i};0}_{0,t})$ such that
$$
%\alpha^k:= 
\phi_i(t,\xi^{\overline k_i}_t)- u(t,\xi^{\overline k_i}_t) =\max_{x_t\in   \Lambda^{0,{\overline k_i};0}_{0,t} } \{  \phi_i(t,x_t)-u(t,x_t) \}\geq \delta\text{ for almost all }\omega\in\overline\Omega_t.
$$ 
   W.l.o.g., we take $\overline\Omega_t=\Omega$ in what follows.

For each $s\in(t,T]$, choose an $\sF_s$-measurable and $ \Lambda^{0,{\overline k_i};\xi^{\overline k_i}_t}_{t,s}$-valued variable $\xi^{\overline k_i}_s$ such that  
\begin{align}
\left(\phi_i(s,\xi^{\overline k_i}_s)- u(s,\xi^{\overline k_i}_s)  \right)^+=\max_{x_s\in \Lambda^{0,{\overline k_i};\xi^{\overline k_i}_t}_{t,s}   }   \left( \phi_i(s,x_s)-u(s,x_s) \right)^+ , \quad \text{a.s.,}\label{eq-maxima}
\end{align}  
and set
%\begin{align*}
$Y^{\overline k_i}(s)
=
	(\phi_i(s,\xi^{\overline k_i}_s)-u(s,\xi^{\overline k_i}_s) )^+ +\frac{\delta (s-t)}{3(T-t)}
	$, and 
$Z^{\overline k_i}(s)
= 
	\esssup_{\tau\in\cT^s} E_{\sF_s} [Y^{\overline k_i}({\tau})]$.
%\end{align*}
Here, recall that  $\mathcal{T}^s$ denotes the set of stopping times valued in $[s,T]$. As $(\phi_i-u)^+\in \cS^2(\Lambda^0;\bR)$, there follows obviously the time-continuity of  
$$
	\max_{x_s\in \Lambda^{0,{\overline k_i};\xi^{\overline k_i}_t}_{t,s}  }   \left(\phi_i(s,x_s)-u(s,x_s) \right)^+, \quad \text{for }s\in[t,T],
$$ 
and thus that of
$\left( \phi_i(s,\xi^{\overline k_i}_s)-u(s,\xi^{\overline k_i}_s) \right)^+$, although the continuity of process $(\xi^{\overline k_i}_s)_{s\in [t,T]}$ (as path space-valued process)  can not be ensured. Therefore, the process $(Y^{\overline k_i}(s))_{t\leq s \leq T}$ has continuous trajectories. 
Define $\tau^{\overline k_i}=\inf\{s\geq t:\, Y^{\overline k_i}(s)=Z^{\overline k_i}(s)\}$. In view of the optimal stopping theory, observe that
$$
E_{\sF_t}\left[Y^{\overline k_i}(T) \right]
%=0
=\frac{\delta}{3}
%	<\alpha^k=
	<\delta \leq 
	Y^{\overline k_i}(t) \leq Z^{\overline k_i}(t)=E_{\sF_t} \left[ Y^{\overline k_i}({\tau^{\overline k_i}}) \right] =E_{\sF_t}\left[Z^{\overline k_i}({\tau^{\overline k_i}})\right],
$$
which gives that $\bP(\tau^{\overline k_i}<T)>0$. As 
\begin{align}
	(\phi_i(\tau^{\overline k_i},\xi^{\overline k_i}_{\tau^{\overline k_i}}) - u(\tau^{\overline k_i},\xi^{\overline k_i}_{\tau^{\overline k_i}}) )^+ 
	+\frac{\delta (\tau^{\overline k_i}-t)}{3(T-t)}
=Z^{\overline k_i}({\tau^{\overline k_i}}) \geq E_{\sF_{\tau^{\overline k_i}}}[Y^{\overline k_i}(T)]= \frac{\delta}{3},
\label{relation-tau}
\end{align}
we have 
\begin{align*}
\bP(( \phi_i(\tau^{\overline k_i},\xi^{\overline k_i}_{\tau^{\overline k_i}}) - u(\tau^{\overline k_i},\xi^{\overline k_i}_{\tau^{\overline k_i}}))^+>0)>0. 
\end{align*}
 Define 
$$\hat\tau^{\overline k_i}=\inf\{s\geq\tau^{\overline k_i}:\, (\phi_i(s,\xi^{\overline k_i}_s)-u(s,\xi^{\overline k_i}_{s})  )^+\leq 0\}.$$ Obviously, $\tau^{\overline k_i}\leq\hat\tau^{\overline k_i}\leq T$.
 Put $\Omega_{\tau^{\overline k_i}}=\{\tau^{\overline k_i}<\hat\tau^{\overline k_i}\}$. Then $\Omega_{\tau^{\overline k_i}}\in \sF_{\tau^{\overline k_i}}$ and in view of relation \eqref{relation-tau}, and the definition of $\hat\tau^{\overline k_i}$, we have $\Omega_{\tau^{\overline k_i}} =\{\tau^{\overline k_i}<T\}$ and  $\bP(\Omega_{\tau^{\overline k_i}})>0$.

Set 
$$\Phi_i(s,x_s)=\phi_i(s,x_s)  +\frac{\delta (s-t)}{3(T-t)}-E_{\sF_s}\left[Y^{\overline k_i}({\tau^{\overline k_i}})\right],\quad s\in[0,T].
$$
 Notice that the process $\left(E_{\sF_s}\left[Y^{\overline k_i}({\tau^{\overline k_i}})\right]\right)_{s\in[0,T]}$ belongs to $\mathscr{C}^1_{\sF}$. Indeed, $E_{\sF_s}\left[Y^{\overline k_i}({\tau^{\overline k_i}})\right]$ is defined on $\Omega\times[0,T]$ independent of $x_s\in\Lambda$ and for $s\leq {\tau^{\overline k_i}}$, the martingale representation under the assumed Brownian filtration $(\sF_r)_{r\geq 0}$ gives the well-defined $\mathfrak{d}_{\omega}\left(E_{\sF_s}\left[Y^{\overline k_i}({\tau^{\overline k_i}})\right]\right)$.
 Thus, $\Phi_i\in\mathscr{C}^1_{\sF}$ since $\phi_i \in \mathscr{C}^1_{\sF}$. For each $\bar\tau\in\cT^{\tau^{\overline k_i}}$, %\footnote{Recall that $\cT^{\tau}$ denotes the set of stopping times $\zeta$ satifying $\tau\leq \zeta\leq T$ as defined in Section 2.2.} 
 we have for almost all $\omega\in\Omega_{\tau^{\overline k_i}}$, 
\begin{align*}
\left(\Phi_i-u\right)(\tau^{\overline k_i},\xi^{\overline k_i}_{\tau^{\overline k_i}})
=0=Z^{\overline k_i}({\tau^{\overline k_i}})-Y^{\overline k_i}({\tau^{\overline k_i}})
&\geq 
E_{\sF_{\tau^{\overline k_i}}}\left[Y^{\overline k_i}({\bar\tau\wedge \hat{\tau}^{\overline k_i}}) \right] - Y^{\overline k_i}({\tau^{\overline k_i}})
\\
&\geq 
 E_{\sF_{\tau^{\overline k_i}}} \left[ \max_{y\in  \Lambda^{0,{\overline k_i};  \xi^{\overline k_i}_{\tau^{\overline k_i}}}_{\tau^{\overline k_i}, \bar\tau\wedge\hat\tau^{\overline k_i}}} (\Phi_i-u)(\bar\tau\wedge\hat\tau^{\overline k_i},y)   \right],
\end{align*}
where we have used the obvious relation $\Lambda^{0,{\overline k_i};  \xi^{\overline k_i}_{\tau^{\overline k_i}}}_{\tau^{\overline k_i}, \bar\tau\wedge\hat\tau^{\overline k_i}} \subset \Lambda^{0,{\overline k_i};\xi^{\overline k_i}_t}_{t, \bar\tau\wedge\hat\tau^{\overline k_i} }$.
This together with the arbitrariness of $\bar\tau$ implies that $\Phi_i\in \overline{\cG} u(\tau^{\overline k_i},\xi^{\overline k_i}_{\tau_{\overline k_i}};\Omega_{\tau^{\overline k_i}},{\overline k_i})$. In view of the correspondence between the viscosity supersolution $u$ and the infinite sequence $\{{\overline k_1}, \overline k_2,\cdots,\overline k_n,\cdots\}$,   we have   for almost all $\omega\in\Omega_{\tau^{{\overline k_i}}}$, 
\begin{align*}
0
&\leq
	 \text{ess}\limsup_{(s,x)\rightarrow ((\tau^{\overline k_i})^+,\,\, \xi^{{\overline k_i}}_{\tau^{{\overline k_i}}})}
	  \left\{ -\mathfrak{d}_{s}\Phi_i(s,x)-\bH(s,x,\nabla \Phi_i(s,x) ) \right\}
\\
&=
	-\frac{\delta}{3(T-t)}
%\\
%&\quad\quad
	+ \text{ess}\limsup_{(s,x)\rightarrow ((\tau^{{\overline k_i}})^+,\,\, \xi^{{\overline k_i}}_{\tau^{{\overline k_i}}})}
	  \left\{  
	-\mathfrak{d}_{s}\phi_i(s,x)-\bH(s,x,\nabla \phi_i(s,x)  ) \right\}
\\
%&\geq 
%	\frac{\alpha}{2(T-t)}
%	+ \text{ess}\liminf_{(s,x)\rightarrow (\tau^+,\xi_{\tau})}
%	E_{\sF_{\tau}} \left\{ -\mathscr L^{\theta_s}J(s,x;\theta)-f(s,x,\theta_s) \right\}
%\\
%&\geq
%	\frac{\kappa}{2(T-t)} + \text{ess}\liminf_{(s,x)\rightarrow (\tau^+,\xi_{\tau})}
%	E_{\sF_{\tau}} \left\{  
%	-\mathfrak{d}_{s}\hat J_l(s,x)-\bH(s,x,D \hat J_l(s,x)  ) \right\}  \\
%	&
%	\quad\quad\quad - \eps E_{\sF_{\tau}}\left[   \sup_{(s,x,v)\in [0,T]\times \bR^d \times v}   \left|\beta_l(s,x,v)\right| \cdot |Dh(x-\bar x)|  \right]	\\
&\leq -\frac{\delta}{3(T-t)} ,
\end{align*}
which is a contradiction.
\end{proof}

\subsection{Uniqueness} \label{subsection-uniqueness}
In addition to Assumption $(\cA 1)$, we assume the joint time-space continuity, i.e.,\\[5pt]
$(\cA 2)$ for each $v\in U$,  $f(\cdot,\cdot,v),\beta^i(\cdot,\cdot,v)\in\cS ^{\infty} (\Lambda; {\mathbb R })$, for $i=1,\dots,d$.\\

 We may approximate the coefficients $\beta,\,f$, and $G$ via regular functions.
\begin{lem}\label{lem-approx}
 Let $(\cA1)$ hold. For each $\eps>0$, there exist partition $0=t_0<t_1<\cdots<t_{N-1}<t_N=T$ for some $N>3$ and functions 
 {\small
$$(G^N,f^N,\beta^N)\in C^{3}(\bR^{ N\times (d+m)+d})\times C( U;C^3([0,T]\times\bR^{(m+d)\times N+d})) \times C(U;C^3([0,T]\times\bR^{(m+d)\times N+d}))$$  
}
such that for each $k\in \bN^+$, 
\begin{align*}
%&
%	G^{\eps}:=\esssup_{x\in\Lambda_T^0} \left|G^N( W({t_1}),\cdots, W({t_N}),x(t_1),\cdots,x(t_N))-G(x) \right|,
%\\
& 
	f^{\eps}_k(t) :=\esssup_{(x,v)\in\Lambda_{0,t}^{0,k}\times U}
		\left|f^N( W({t_1\wedge t}),\ldots, W({t_N\wedge t}),t,x({t_0\wedge t}),\ldots, x({t_N\wedge t}),v)-f(t,x,v)\right|,
\\
&
 	\beta^{\eps}_k(t) :=\esssup_{(x,v)\in \Lambda_{0,t}^{0,k} \times U}
		\left|\beta^N( W({t_1\wedge t}),\ldots, W({t_N\wedge t}),t,x({t_0\wedge t}),\ldots, x({t_N\wedge t}),v)-\beta(t,x,v)\right|,
\end{align*}
and $G^{\eps}_k:=\esssup_{x\in\Lambda_T^{0,k}} \left|G^N( W({t_1}),\ldots, W({t_N}),x(t_0),\ldots,x(t_N))-G(x) \right|$
 are $\sF_t$-adapted with
\begin{align}
	\left\| G^{\eps}_k  \right\|_{L^2(\Omega,\sF_T;\bR)} + \left\| f^{\eps}_k  \right\|_{L^2(\Omega\times[0,T];\bR)}   + \left\| \beta^{\eps}_k  \right\|_{L^2(\Omega\times[0,T];\bR)}    <\eps(1+k).
	\label{approx-error}
\end{align}
Moreover, $G^N$, $f^N$, and $\beta^N$ are uniformly Lipschitz-continuous in the space variable $x$ with an identical Lipschitz-constant $L_c$ independent of $N$, $k$, and $\eps$.
\end{lem}
  Lemma \ref{lem-approx} may be proved with density arguments that are more or less standard; the sketch of the proof is given for the reader's reference.
  
  \begin{proof}[Sketched proof of Lemma \ref{lem-approx}]
  We consider the approximations for the function $f$.  First,
  % in view of assumption $(\cA 1)$, we may embed $H^{1,\infty}$ into some Hilbert spaces with weights as in \cite[Theorem 2.2]{GraeweHorstQui13}, and 
  %in a similar way to \cite[case (c) in the proof of Proposition 2.2, Page 29]{gawarecki2010stochastic}, 
  the joint time-space continuity in Assumption $(\cA 2)$ and the dominated convergence theorem allow us to approximate $f$ via random functions of the form:
  \begin{align}
   \bar f^l(\omega,t,x,v)= f(\omega,0,x_0,v)1_{[0,t_1]}(t)+\sum_{j=1}^{l-1}f(\omega,t_{j},x_{t_{j}},v) 1_{(t_{j},t_{j+1}]}(t), \quad
  t\in[0,T], \label{approximate-1}
  \end{align}
  where $0=t_0< t_1< \cdots< t_l<T$. %Here, straightfo the above approximation
  % and for $j=1,\dots,l$, $\phi_j=f(t_{j-1},\cdot,\cdot,\cdot)\in L^2(\Omega,\sF_{t_{j-1}}; C (U\times\bR^d))$. 
  
%  In fact, with the identity approximations as in \eqref{appr-l}, we may take instead 
%  $$\phi_j\in L^2(\Omega,\sF_{t_{j-1}};
%   C (U, C^{\infty}(\bR^d)),\quad j=1,\dots, l.$$ 
   
  Next, for each $j\geq 0$, the function $f(\omega,t_{j},x_{t_j},v)$ may be approximated \textit{monotonically} (see \cite[Lemma 1.2, Page 16]{da2014stochastic} for instance) by simple random variables of the following form:
  $$
  \sum_{i=1}^{l_j}1_{A_i^j} (\omega) h^j_i(x_{t_j},v),\quad \text{with }h_i^j\in C(\Lambda_{t_j}\times U),\quad A_i^j\in\sF_{t_{j}}, \quad i=1,\ldots, l_j,
  $$
  and  by \cite[Lemma 4.3.1., page 50]{oksendal2003stochastic}, each $1_{A_i^j}$ may be approximated in $L^2(\Omega,\sF_{t_{j}})$ by functions in the following set
  $$
  \{g(W({\tilde t_1}),\ldots,W({\tilde t_{l^j_i}}) ):\, g\in  C^{\infty}_c(\bR^{l_i^j\times m}),\,\tilde t_r\in[0,t_{j}],\,r=1,\ldots,l_i^j \},
  $$ 
  where $C^{\infty}_c(\bR^{l_i^j\times m})$ denotes the space of infinitely differentiable functions with compact supports in $ \bR^{l_i^j\times m}$.  Moreover, recalling that for each $x_{t_j}\in \Lambda^0_{t_j}$, 
  \begin{align}
 &\lim_{M\rightarrow \infty} \max_{s\in [0,t_j]} 
 \left|
 	x_{t_j}(s)- P^M(x_{t_j})(s) \right|=0, \text{ with }  \label{approximate-2}
\\
&
	P^M(x_{t_j})(s) := \sum_{n=1}^{2^M} x_{t_j}\left(\frac{(n-1)t_j}{2^M}\right)1_{[\frac{(n-1)t_j}{2^M},\frac{nt_j}{2^M})} (s) 
	+x_{t_j}(t_j)1_{\{t_j\}}(s), \quad M\in\bN^+, \nonumber
  \end{align}
  we may have each function $h^j_i(x_{t_j},v)$ be approximated by 
  $$\tilde{h}^j_i \left(x_{t_j}(0),\ldots,x_{t_j}\left(\frac{(2^M-1)t_j}{2^M}\right),x_{t_j}(t_j),v\right):   =h^j_i(P^M(x_{t_j}),v),$$
  and as a continuous function lying in $C(\bR^{(2^M+1)\times d}\times U)$,  each function $\tilde h^j_i$ may be approached by infinitely differentiable functions (denoted by itself) lying in $C( U; C^{\infty}(\bR^{(2^M+1)\times d}))$ with a uniform identical Lipschitz constant w.r.t. the space variable $x$. In addition, each $1_{(t_{j-1},t_j]}$ may be increasingly approximated by compactly-supported nonnegative functions $\varphi_j\in C^{\infty}((t_{j-1},T];\bR)$. 
  
  To sum up, we may put all the partitions together, and the function $f$ may be approximated  by random functions of the following form:
  \begin{align}
  &f^N( W({\bar t_1\wedge t}),\ldots, W({\bar t_N\wedge t}), ,t,x(0), x({\bar t_1\wedge t}),\ldots, x({\bar t_N\wedge t}),v)
  \notag\\
  &=\sum_{j=1}^{\bar l}\sum_{i=1}^{\bar l_j} g_i^j(W({\bar t_1}),\ldots, W({\bar t_j }))\tilde h_i^j(x(0), x({\bar t_1 }),\ldots, x({\bar t_j}),v) \varphi_j(t), \label{approx-f}
    \end{align}
  where $0=\bar t_0 <\bar t_1<\cdots<\bar t_{N-1}<\bar t_{N}=T$, and $g_i^j,\,\tilde h_i^j (\cdot, v),\, \varphi_j$ are smooth functions\footnote{When $t\in (\bar t_j,\bar t_{j+1})$, we write the dependence of $f^N$ on $W(t)$ and $x(t)$ just for notational convenience, though the defined function $f^N$ does not nontrivially depend on $W(t)$ and $x(t)$ for $t\in (\bar t_j,\bar t_{j+1})$; similarly,  the functions $g_i^j$ and $\tilde h_i^j$ may not nontrivially depend on some particular input(s) in expression \eqref{approx-f}. }. The required approximations for $G$ and $\beta$ are following in a similar way. Just note that the approximation error in \eqref{approx-error} is given as $\eps(1+k)$ depending on $k$ due to the approximations of  the paths in Lipschitz-continuous function spaces involved in \eqref{approximate-1} and \eqref{approximate-2}.
  \end{proof}

\begin{thm}\label{thm-main}   
 Let Assumptions $(\cA 1)$ and $(\cA 2)$ hold. The viscosity solution to SPHJ equation \eqref{SHJB} is unique.
 \end{thm}

\begin{proof} %[Proof of Theorem \ref{thm-main}]
For each $k\in \bN^+$, define 
\begin{align*}
\overline{\mathscr V}_k=
\bigg\{
	\phi\in\mathscr C^1_{\sF}:\, \,\phi(T,x)\geq G(x)\,\,\forall x\in\Lambda_{T}^{0}, \text{ a.s.,  and for each }t\in[0,T) \text{ with }y\in \Lambda_{0,t}^{0,k},
&\\
	\text{ess}\liminf_{(s,x)\rightarrow (t^+,y)}	 \left[ -\mathfrak{d}_{s}\phi(s,x)-\bH(s,x,\nabla \phi(s,x) ) \right]
	\geq 0,	\quad \text{a.s.}
 &\,\,  \bigg\},\\
\underline{\mathscr V}_k=
\bigg\{
	\phi\in\mathscr C^1_{\sF}:\, \phi(T,x)\leq G(x)\,\,\,\forall x\in\Lambda_{T}^{0}, \text{ a.s., and for each }t\in[0,T) \text{ with }y\in \Lambda_{0,t}^{0,k},
	&\\
		\text{ess}\limsup_{(s,x)\rightarrow (t^+,y)}   \left[ -\mathfrak{d}_{s}\phi(s,x)-\bH(s,x,\nabla\phi(s,x) ) \right] 
		\leq 0,\quad \text{a.s.}
	&\,\,   \bigg\},
\end{align*}
and set
\begin{align*}
\overline{u}_k=\essinf_{\phi_k\in \overline{\mathscr V}_k} \phi_k, \quad 
\underline{u}_k=\esssup_{\phi_k\in \underline{\mathscr V}_k} \phi_k.
\end{align*}
It is easy to see the monotonicity of $\overline {\mathscr V}_k$, $\underline{\mathscr V}_k$, and thus that of $\overline u_k$ and $\underline u_k$, as $k\rightarrow \infty$, and we may define further
$$
\overline u =\lim_{k\rightarrow \infty} \overline u_k,\quad \underline u=\lim_{k\rightarrow \infty} \underline u_k.
$$

 By the comparison principle of Proposition \ref{prop-cor-cmp}, each viscosity solution $u$ satisfies $\underline u\leq u\leq \overline u$ on $\cup_{k=1}^{\infty} \Lambda_{0,T}^{0,k}$ that is dense in $\Lambda_T^0$.  Therefore, it is sufficient to verify $\underline u=V= \overline u$ for the uniqueness of viscosity solution.    The proof will be divided into two steps.

\textbf{Step 1.} We construct functions from $\overline{\mathscr V}_k$ and $\underline{{\mathscr V}}_k$ to dominate the value function $V$ from above and from below respectively. Let $(\Omega',\sF',\{\sF'_t\}_{t\geq 0}, \bP')$ be another complete filtered probability space on which a d-dimensional standard Brownian motion $B=\{B(t)\, :\, t\geq 0\}$ is well defined.  The filtration $\{\sF'_t\}_{t\geq 0}$ is generated by $B$ and augmented by all the $\bP'$-null sets in $\sF'$. Put
$$
(\bar\Omega,\bar\sF,\{\bar\sF_t\}_{t\geq 0},\bar\bP)=
(\Omega\times\Omega',\sF \otimes\sF',\{\sF_t\otimes\sF'_t \}_{t\geq 0},\bP \otimes\bP'),
$$
and denote by $\bar\cU$ the set of all the $U$-valued and $\bar\sF_t$-adapted processes. Then we have two Brownian motions $B$ and $W$ that are independent on $(\bar\Omega,\bar\sF,\{\bar\sF_t\}_{t\geq 0},\bar\bP)$, and all the theory established in previous sections still hold on the enlarged probability space.

Fix an arbitrary $\eps\in(0,1)$ and $k\in \bN^+$, and choose $(G^{\eps}_k,\,f^{\eps}_k,\,\beta^{\eps}_k)$ and $(G^N,f^N,\beta^N)$ as in Lemma \ref{lem-approx}. By the theory of backward SDEs (see \cite{Hu_2002} for instance), let the  pairs  $(Y_k^{\eps},Z_k^{\eps})$ and $(y,z )$   be the unique adapted solutions  to backward SDEs
$$
Y_k^{\eps}(s)=G_k^{\eps}+
	\int_s^T\left(f_k^{\eps}(t)+C_1\beta_k^{\eps}(t)\right)\,dt
		-\int_s^TZ_k^{\eps}(t)\,d W(t),
$$
 and 
 $$
 y(s)=
 %\sup_{0\leq t\leq T}|B_t| 
 \|B_T\|_0+\int_{s}^T \|B_r\|_0 \,dr-\int_s^T z(r)\,dB(r),
 $$
respectively, with the constant $C_1\geq 0$ to be determined later. For each $s\in[0,T)$ and $x_s\in \Lambda_s$,  set
{\small
\begin{align*}
{V}^{\eps}(s,x_s)
=\essinf_{\theta\in\bar\cU} E_{\bar\sF_s} \bigg[
	\int_s^Tf^N\Big( & W({t_1\wedge t}),\ldots, W({t_N\wedge t}),t,X^{s,x_s;\theta,N}(0),X^{s,x_s;\theta,N}({t_1\wedge t}) ,\ldots, 
	\\
	&X^{s,x_s;\theta,N}({t_N\wedge t}), \theta(t)\Big)\,dt\\
		+G^N\Big( &W({t_1}),\ldots, W({t_N}),,X^{s,x_s;\theta,N}(0),X^{s,x_s;\theta,N}({t_1}),\ldots,X^{s,x_s;\theta,N}(t_N)\Big)
				\bigg],
\end{align*}
}
where $X^{s,x_s;\theta,N}(t)$ satisfies the SDE
\begin{equation*}
\left\{
\begin{split}
dX(t)&=\beta^N(W({t_1\wedge t}),\ldots, W({t_N\wedge t}),t,X (0),X ({t_1\wedge t}) ,\ldots,X ({t_N\wedge t}),\theta(t))\, dt 
\\
&\quad 
	+\delta \, dB(t),\,\,\,t\in[s,T]; \\
 X(t)&=x_s(t),\quad t\in[0,s],
\end{split}
\right.
\end{equation*}
with $\delta\in (0,1)$ being an arbitrary constant.

For each $s\in[t_{N-1},T)$, let
$$V^{\eps}(s,x_s)=\tilde V^{\eps}( W(t_1),\ldots, W({t_{N-1}}),{W}(s),s,x(0),\ldots,x(t_{N-1}),x(s))$$ 
with
{\small
\begin{align*}
&\tilde V^{\eps}( W(t_1),\ldots, W({t_{N-1}}),\tilde y,s,x(0),\ldots,x(t_{N-1}),\tilde x) \\
&=\essinf_{\theta\in\bar\cU} E_{\bar\sF_s,  W(s)=\tilde y,x_s(s)=\tilde x} \bigg[
	\int_s^Tf^N\left(  W({t_1}),\ldots, W({t_{N-1}}), W({ t}),t,\ldots,x(t_{N-1}),X^{s,x_s;\theta,N}(t),\theta(t)\right)\,dt\\
	&
		+G^N\left( W({t_1}),\ldots, W({t_N}),x(0),\ldots,x(t_{N-1}),X^{s,x_s;\theta,N}(T)\right)
		\bigg].
\end{align*}
}
Here and in what follows, we just write $x(t_j)=x_s(t_j)$ for $j=0,\ldots,N-1$, as they are deemed to be fixed for $s\in (t_{N-1},T]$.
 By the viscosity solution theory of fully nonlinear parabolic PDEs (see \cite[Theorems I.1 and II.1]{lions-1983} for instance),  the function 
 $$\tilde V^{\eps}( W(t_1),\cdots, W({t_{N-1}}),\tilde y,s,x(0),\cdots,x(t_{N-1}),\tilde x)$$
 %of $(\tilde y, s,\tilde x)\in \bR^d\times [t_{N-1},T]\times \bR^d$
%  when $s\in[t_{N-1},T)$, 
%$$V^{\eps}(s,x_s)=\tilde V^{\eps}( W(t_1),\cdots, W({t_{N-1}}),{W}(s),s,x(0),\cdots,x(t_{N-1}),x(s))$$ 
%with
%{\small
%\begin{align*}
%&\tilde V^{\eps}( W(t_1),\cdots, W({t_{N-1}}),\tilde y,s,x(0),\cdots,x(t_{N-1}),\tilde x) \\
%&=\essinf_{\theta\in\cU} E_{\bar\sF_s,  W(s)=\tilde y,x_s(s)=\tilde x} \bigg[
%	\int_s^Tf^N\left(  W({t_1}),\cdots, W({t_{N-1}}), W({ t}),t,\cdots,x(t_{N-1}),X^{s,x_s;\theta,N}(t),\theta(t)\right)\,dt\\
%	&
%		+G^N\left( W({t_1}),\cdots, W({t_N}),x(0),\cdots,x(t_{N-1}),X^{s,x_s;\theta,N}(T)\right)
%		\bigg],
%\end{align*}
%}
%as a function of $(\tilde y, s,\tilde x)$ on $ \bR^d\times [t_{N-1},T]\times \bR^d$, 
satisfies the  following HJB equation:
{\small
\begin{equation}\label{HJB-N}
  \left\{\begin{array}{l}
  \begin{split}
  -D_tu(\tilde y, t,\tilde x)=\,& %\frac{1}{2}\text{tr}\left\{\left(+ D^2u(t,x)\right)\right\}
  \frac{1}{2} \text{tr}\left(D_{\tilde y \tilde y}u(\tilde y, t,\tilde x)\right) 
  +\frac{\delta^2}{2} \text{tr}\left( D_{\tilde x \tilde x}u(\tilde y, t,\tilde x)\right) \\
  +&\inf_{v\in U} \bigg\{   
%	  \text{tr}\Big(  \frac{1}{2}\sigma \sigma'( \tilde W_{t_1},\cdots,\tilde W_{t_{N-1}},y,t,x,v) D^2_{xx}u(t,x,y)
%	  \\
%	  &+\tilde  \sigma  ( \tilde W_{t_1},\cdots,\tilde W_{t_{N-1}},y,t,x,v)D^2_{xy}u(t,x,y)
%	  \Big)\\
%	&
		  (\beta^N)'(  W({t_1}),\ldots, W({t_{N-1}}) ,\tilde y,t,  x(0),\ldots,x(t_{N-1}),\tilde x,v)D_{\tilde{x}}u(\tilde y, t,\tilde x)\\
	  &+f^N(  W({t_1}),\ldots, W({t_{N-1}}) ,\tilde y,t,  x(0),\ldots,x(t_{N-1}),\tilde x,v)\bigg\}
                     ;\\
    u(\tilde y, T,\tilde x)=\, &G^N(  W({t_1}),\ldots, W({t_{N-1}}) ,\tilde y,  x(0),\ldots,x(t_{N-1}),\tilde x).
        \end{split}
  \end{array}\right.
\end{equation}
}
Thus, the regularity theory of viscosity solutions (see \cite[Theorem 6.4.3]{krylov1987nonlinear} for instance\footnote{As $U\subset\bR^n$ is a nonempty separable set, it has a denumerable subset $\mathcal K\subset U$ that is dense in $U$, and by the continuity of the coefficients, the essential infimum may be taken over $\mathcal K$. Thus,  we may apply  \cite[Theorem 6.4.3]{krylov1987nonlinear} straightforwardly.}) 
 gives for each $(x(0),\ldots,x(t_{N-1}))\in\bR^{N\times d}$,
\begin{align*}
\tilde V^{\eps}&( W(t_1),\ldots, W({t_{N-1}}),\cdot,\cdot,x(0),\ldots,x(t_{N-1}),\cdot)
\\
&\in 
%L^{\infty}\left(\Omega,\sF_{t_{N-1}}; C^{1+\frac{\bar\alpha}{2},2+\bar\alpha}([t_{N-1},T]\times\bR^{d+m}) \right) 
\cap_{\bar t\in (t_{N-1},T)} L^{\infty}\left(\Omega, \sF_{t_{N-1}}; C^{1+\frac{\bar\alpha}{2},2+\bar\alpha}([t_{N-1},\bar t\, ]\times\bR^{m+d}) \right) ,
\end{align*}
for some $\bar\alpha \in (0,1)$, where the \textit{time-space} H\"older space $C^{1+\frac{\bar\alpha}{2},2+\bar\alpha}([t_{N-1},\bar t]\times\bR^{d+m})$ is defined as usual. Similar arguments may be conducted on time interval $[t_{N-2},t_{N-1}]$ with the previously obtained $V^{\eps}(t_{N-1},x)$ as the terminal value, and recursively on intervals $[t_{N-3},t_{N-2}]$, $\dots$, $[0,t_{1}]$.

On $[t_{N-1},T]$, applying the  It\^o-Kunita formula of \cite[Pages 118-119]{kunita1981some} to $\tilde{V}^{\eps} $ yields that
 {\small
 \begin{equation}\label{SHJB-N}
  \left\{\begin{array}{l}
  \begin{split}
  &-dV^{\eps}(t,(x-\delta B)_t)\\ 
&=  \inf_{v\in U} \bigg\{
  (\beta^N)'( W({t_1}),\ldots, W(t),t,x(0),\ldots,x(t_{N-1})-\delta B(t_{N-1}),x(t)-\delta B(t),v) \nabla V^{\eps}(t,(x-\delta B)_t)\\
&\quad	
  +f^N( W({t_1}),\ldots, W(t),t,x(0),\ldots,x(t_{N-1})-\delta B(t_{N-1}),x(t)-\delta B(t),v)\bigg\}\,dt\\
&\quad
		-D_{\tilde y} \tilde{V}^{\eps}( W({t_1}),\ldots, W(t),t,x(0),\ldots,x(t_{N-1})-\delta B(t_{N-1}),x(t)-\delta B(t)) \,d W(t) \\
&\quad
		+\delta \nabla V^{\eps}(t,(x-\delta B)_t)\,dB(t),
                     \quad t\in [t_{N-1},T) \text{ and }x\in\Lambda_t;\\
    &V^{\eps}(T,x_T)=\, G^N(  W({t_1}),\cdots, W(T),x(0),\cdot,x(t_{N-1})-\delta B(t_{N-1}),x(T)-\delta B(T)), \quad x_T\in\Lambda_T.
    \end{split}
  \end{array}\right.
\end{equation}
}
It follows similarly on intervals $[t_{N-2},t_{N-1})$, $\dots$, $[0,t_{1})$, and finally we have $V^{\eps}(\cdot,\cdot-\delta B({\cdot})) \in \mathscr C^1_{\bar\sF}$.

In view of the approximation in Lemma \ref{lem-approx} and with an analogy to  (iv) in Proposition \ref{prop-value-func}, there exists $\tilde{L}>0$ such that for all $t\in [0,T]$ with $x_t\in\Lambda_t$,
$$
 |\nabla V^{\eps}(t,x_t)|   \
\leq \tilde L,\,\,\,\text{a.s.,}
$$
with $\tilde L$ independent of $\eps$ and $N$. Set $C_1=\tilde L$, and
\begin{align*}
\overline{V}_k^{\eps}(s,x)
&=
	V^{\eps}(s,(x-\delta B)_s)+Y_k^{\eps}(s)+\delta C_2 y(t),\\
\underline{V}_k^{\eps}(s,x)
&=
	V^{\eps}(s,(x-\delta B)_s)-Y_k^{\eps}(s)-\delta C_2 y(t),
\end{align*}
where $C_2=4L_c(\tilde L+1)$ and $L_c$ is the Lipschitz constant in Lemma \ref{lem-approx}.
%with $\bar K$ determined later.
As the uniform Lipschitz continuity gives
\begin{align*}
|\beta^N(t,x_t,v)-\beta^N(t,(x-\delta B)_t,v)|
+|f^N(t,x_t,v)-f^N(t,(x-\delta B)_t,v)| 
& \leq 2\delta L_c \|B_t\|_0,\\
|G^N(x_T)-G^N((x-\delta B)_T)|
&\leq \delta L_c \|B_T\|_0,
\end{align*}
it holds that for all $(t,x_t)$ with $t\in (t_{N-1},T)$ and $x_t\in \Lambda_{0,t}^{0,k}$, 
\begin{align}
&-\mathfrak{d}_{t}\overline V_k^{\eps}-\bH(\nabla \overline V_k^{\eps})
\nonumber\\
&=
	-\mathfrak{d}_{t}\overline V_k^{\eps}
	-\inf_{v\in U} \bigg\{
%	  \text{tr}\Big(  \frac{1}{2}\sigma \sigma' D^2\overline V^{\eps}
%	  +  \sigma D\mathfrak{d}_{\omega}\overline V^{\eps}
%	  \Big)
  		     (\beta^N)'\nabla \overline V_k^{\eps}
	  +f^N+f_k^{\eps}+ \tilde L  \beta_k^{\eps}
	  +\delta C_2 \|B_t\|_0 \nonumber\\
&
	  \quad\quad
	  +   \left(\beta-\beta^N\right)' \nabla \overline V_k^{\eps}       -\beta_k^{\eps} \tilde L
	  +f-f^N-f_k^{\eps}-\delta C_2 \|B_t\|_0
	   \bigg\}
\nonumber\\
&\geq
	-\mathfrak{d}_{t}\overline V_k^{\eps}
	-\inf_{v\in U} \bigg\{
%	  \text{tr}\Big(  \frac{1}{2}\sigma \sigma' D^2\overline V^{\eps}
%	  +  \sigma D\mathfrak{d}_{\omega}\overline V^{\eps}
%	  \Big)
  		      (\beta^N)'\nabla \overline V_k^{\eps}
	  +f^N+f_k^{\eps}+\beta_k^{\eps} \tilde L 
	  +\delta C_2 \|B_t\|_0
	   \bigg\}
	 \label{est-A4}\\
&=0,\nonumber
\end{align}
where the inputs are omitted for involved functions.
It follows similarly on intervals $(t_{N-2},t_{N-1})$, $\dots$, $[0,t_1)$ that
$
-\mathfrak{d}_{t}\overline V_k^{\eps}-\bH(\nabla \overline V_k^{\eps} ) \geq 0,
$
which together with the obvious relation  $\overline V^{\eps}_k(T)=G^{\eps}_k+G^N+ \delta C_2 \|B_T\|_0 \geq G$  indicates that $\overline V_k^{\eps}\in \overline {\mathscr V}_k$. Analogously, $\underline V_k^{\eps}\in \underline {\mathscr V}_k$.

\textbf{Step 2.}
Taking $k\geq L$, we measure the distance between $\underline V^{\eps}_k$ (resp., $\overline V^{\eps}_k$) and $V$. By the estimates for solutions of backward SDEs (see \cite[Proposition 3.2]{Hu_2002} for instance), we first have 
\begin{align*}
\|Y_k^{\eps}\|_{L^2(\Omega;C([0,T];\bR))} + \|Z_k^{\eps}\|_{L^2(\Omega\times[0,T];\bR^{m})}
&
	\leq  C_3 \left(  \|G_k^{\eps}\|_{L^2(\Omega,\sF_T;\bR)} + \|f_k^{\eps}+\tilde L \beta_k^{\eps}\|_{L^2(\Omega\times[0,T];\bR^{m})}  \right)
\\
&
	\leq  C_3 \eps \cdot (1+\tilde L) (1+k),
\end{align*}
with the constant $C_3$ independent of $N$, $k$, $\delta$, and $\eps$.  

%$**************\text{Carefully change!!}*******************************$

Fix an arbitrary $s\in[0,T)$ with $x_s\in \Lambda_{0,s}^{0,k}$. As $k\geq L$, we have
$$X^{s,x_s;\theta,N,\delta B}(\cdot):=(X^{s,x_s;\theta,N}(\cdot) -\delta B(\cdot)+\delta B(s)) \cdot \textbf{1}_{[s,T]} (\cdot) + x_s(\cdot)\textbf{1}_{[0,s)}(\cdot)  \in \Lambda_{0,T}^{0,k},$$
and for $0\leq s\leq t\leq T$,  it holds that
\begin{align}
&\left|\beta^N(t,X^{s,x_s;\theta,N}_t,\theta(t)) -\beta(t,X^{s,x_s;\theta}_t,\theta(t)) \right|
\notag\\
&\leq
\left|\beta^N(t,X^{s,x_s;\theta,N}_t,\theta(t)) -\beta^N(t,X^{s,x_s;\theta,N,\delta B}_t,\theta(t)) \right|
+    \beta^{\eps}_k(t)
\\
&\quad
+	\left|      \beta(t,X^{s,x_s;\theta,N,\delta B}_t,\theta(t)) -\beta(t,X^{s,x_s;\theta,N}_t,\theta(t)) \right|
\notag\\
&\quad 
	+\left| \beta(t,X^{s,x_s;\theta,N}_t,\theta(t)) -\beta(t,X^{s,x_s;\theta}_t,\theta(t)) \right|
\notag\\
&\leq 2 (L_c+L)  \delta \|B^{s,0}_t\|_{0}  + \beta_k^{\eps}(t) + L \| X^{s,x_s;\theta,N}_t - X^{s,x_s;\theta}_t \|_0, \quad\text{a.s.,}
\label{eq-beta}
\end{align}
where $B^{s,0}_t(r)=B(r)-B(r\wedge s)$ for $0\leq r\leq t$.
  In view of the approximations in Lemma \ref{lem-approx}, using It\^o's formula, Burkholder-Davis-Gundy's inequality, and Gronwall's inequality, we have through standard computations that for any $\theta\in\cU$,
\begin{align*}
&
	E_{\sF_s} \left [  \sup_{s\leq t\leq T} \left|  X^{s,x_s;\theta,N}(t)-X^{s,x_s;\theta}(t)   \right| ^2  \right]
%\\
%&\leq
%	\tilde K\left(\delta ^2+ E_{\sF_s}\int_s^T\left| \beta^N\left(\tilde W_{t_1\wedge t},\cdots,\tilde W_{t_N\wedge t},t,X^{s,x;\theta,N}_t,\theta_t\right)-\beta\left(t,X^{s,x;\theta,N}_t,\theta_t\right) \right|^2\,dt\right)
%	\\
%&
\leq
C_4 \left( \delta^2+  E_{\sF_s}\int_s^T   \left|\beta_k^{\eps}(t)\right|^2\,dt\right) ,
\end{align*}
with $C_4$ being independent of $s,\, x_s,\,\delta,\,N,k$, $\eps$, and $\theta$. This together with similar calculations as in  \eqref{eq-beta} yields
\begin{align*}
&E\left| V^{\eps}(s,x_s)-V(s,x_s)\right|
\\
&	
	\leq
		E\esssup_{\theta\in\bar\cU}
		E_{\sF_s}\bigg[ 
		\int_s^T\Big(  \Big| f^N\left(t,X^{s,x_s;\theta,N}_t,\theta(t)\right) -f\left(t,X^{s,x_s;\theta}_t,\theta(t)\right)\Big| \Big)\,dt
\\
&\quad\quad\quad
		 +\Big| G^N\left(X^{s,x_s;\theta,N}_T\right) -G\left(X^{s,x_s;\theta}_T \right)\Big|
		\bigg] 
\\
%&
%	\leq E|Y^{\eps}(s)|
%		+
%		C \left( \delta + E \esssup_{\theta\in\cU}
%		\left(
%		E_{\sF_s}\int_s^T   \left|\beta^{\eps}(t)\right|^2\,dt
%		 \right)^{1/2}\right)
%\\
&
	\leq  C \left( \delta  + \left\|  \beta_k^{\eps} \right\|_{L^2(\Omega\times[0,T];\bR^d)} + \left\|  f_k^{\eps} \right\|_{L^2(\Omega\times[0,T];\bR)} +\left\|  G_k^{\eps} \right\|_{L^2(\Omega;\bR)}  \right)
\\
&
	\leq C_5 (\eps (1+k)+\delta),
\end{align*}
with the constant $C_5$ being independent of $N$, $\eps$, $k$, $\delta$, and $(s,x_s)$. Furthermore, in view of the definitions of $\overline V^{\eps}_k$ and $\underline V^{\eps}_k$, there exists some constant $C_6$ independent of $\eps,\,\delta,\,k$, and $N$ such that for all $s\in[0,T]$ and $x_s\in \Lambda_{0,s}^{0,k}$, it holds that
\begin{align}
E\left| \overline V^{\eps}(s,x_s)-V(s,x_s)\right|
+E\left| \underline V^{\eps}(s,x_s)-V(s,x_s)\right|
\leq C_6 \left\{\eps(1+k)+\delta \right\}. \label{est-v-esp-v}
\end{align}
Because $V$ is a viscosity solution by Theorem \ref{thm-existence}, there exist two infinite sequence of integers $\{\overline k_n\}_{n\in\bN^+}$ and $\{\underline k_n\}_{n\in\bN^+}$ (see Remark \ref{rmk-defn}) such that 
$\lim_{n\rightarrow \infty} \overline k_n =\lim_{n\rightarrow \infty} \underline k_n=\infty$, 
$\overline V^{\eps}_{\overline k_n} (s,x_s)\geq V(s,x_s) $ a.s. for all $n\in\bN^+$, $s\in[0,T]$, and $x_s\in \Lambda_{0,s}^{0,\overline k_n}$, and $  V(t,x_t) \geq \underline V^{\eps}_{\underline k_n}(t,x_t)$ for all $n\in\bN^+$, $t\in[0,T]$, and $x_t\in \Lambda_{0,t}^{0,\underline k_n}$. These, together with the arbitrariness of $\eps$, $k$, and $\delta$ in \eqref{est-v-esp-v} and the denseness of $\cup_{n=1}^{\infty} \Lambda_{0,T}^{0,\overline k_n}$ and $\cup_{n=1}^{\infty} \Lambda_{0,T}^{0,\underline k_n}$ in $\Lambda_T^0$,  finally imply that $\underline u(t,x_t)=V(t,x_t)= \overline u(t,x_t)$ a.s. for all $t\in[0,T]$, and $x_t\in\Lambda_t^0$.
\end{proof}
%
%\begin{rmk}\label{rmk-superparab}
The above proof is inspired by but different from the conventional Perron's method and its modifications used, for instance, in \cite{buckdahn2015pathwise,ekren2016viscosity-1,qiu2017viscosity}; the key difference is that the random fields $\overline u$ and $\underline u$ are neither extreme points of viscosity semi-solutions nor from approximate extremality of classical semi-solutions, while they are limits of approximate extreme points of certain classes of regular random functions. Besides,  by enlarging the original filtered probability space with an independent Brownian motion $B$,  we have actually constructed the regular approximations of $V$ with a regular perturbation induced by $\delta  B$, which corresponds to approximation to the optimization \eqref{Control-probm}-\eqref{state-proces-contrl} with stochastic controls of Markovian type. Such approximation seems interesting even for the case where all the coefficients $\beta,\,f,$ and $G$ are just deterministic and path-dependent. 
%\end{rmk}

\begin{rmk}
In this work, the filtration $\{\sF_t\}_{t\geq0}$ satisfying the usual conditions is generated by an $m$-dimensional Wiener process $W=\{W(t):t\in[0,\infty)\}$ and augmented by all the $\bP$-null sets in $\sF$. Such a Brownian filtration assumption is preferred and adopted here due to technical reason. On the one hand, in the above proof of the uniqueness, the constructed approximations squeezing the viscosity solution are based on Markovian-type optimal controls and their sufficient regularity estimate is given in Krylov's work (see \cite[Theorems I.1 and II.1]{lions-1983} for instance) where the Brownian filtration is assumed. On the other hand, with the Brownian filtration, we could relate the random test functions in $\mathscr C_{\sF}^1$ to an It\^o process in Definition \ref{defn-testfunc} with the operators $\mathfrak{d}_t$ and $\mathfrak{d}_{\omega}$ defined conveniently, and  the generalized It\^o-Kunita formula and the existence arguments follow smoothly. Even so, we think the filtration assumption could be loosened. Nevertheless, the extension may not be straightforward. For example, working under an arbitrary filtration satisfying the usual conditions, one could replace the stochastic integral representation in \eqref{derivative} with a parameterized martingale $\cM^u(t,x_t)$ that is linear w.r.t. $u$, then one would need to contemplate how to make sense of the composition $\cM^u(t,X^{{\varrho},x_{\varrho};\theta}_t)$ when dealing with the existence, and in the uniqueness part, the associated differential operator needs to be specified in the Markovian approximations which is hard without any pre-specification on the underlying stochastic process.
\end{rmk}

\begin{appendix}
\section{Proof of Theorem \ref{thm-DPP}}\label{appdx-proof}

\begin{proof}[Proof of Theorem \ref{thm-DPP}]
Denote the right hand side by $\overline V(\tau,\xi)$. By Proposition \ref{prop-value-func} (iv) and (v), both $V$ and $\overline V$ are lying in $\cS^{\infty}(\Lambda;\bR)$ and the continuity indicates that it is sufficient to prove Theorem 3.3 when $\tau,\hat\tau$ and $\xi$ are deterministic. 

%First, the definition for $V(t,x)$ implies that $\overline V(\tau,\xi)\leq V(\tau,\xi)$.  

For each $\eps>0$, by Proposition \ref{prop-value-func} (iv), there exists $\delta=\eps/L_V>0$ such that whenever $\|x_{\hat \tau}-y_{\hat \tau}\|_{0}<\delta$ for $x_{\hat \tau},\,y_{\hat \tau}\in \Lambda_{\hat\tau}$, it holds that
$$
|J(\hat\tau,x_{\hat \tau};\theta)- J(\hat\tau,y_{\hat \tau};\theta)  | + 
|V(\hat\tau,x_{\hat \tau})-V(\hat\tau,y_{\hat \tau})|
\leq \eps \quad \text{a.s., }\forall \, \theta\in\cU.
$$
Arzel$\grave{\text{a}}$-Ascoli theorem indicates the compactness of $\Lambda^{0,L;\xi}_{\tau,\hat\tau}$ in $\Lambda_{\hat\tau}^0$. Thus, $\Lambda^{0,L;\xi}_{\tau,\hat\tau}$ is separable, and there exists a sequence $\{x^j\}_{j\in\bN^+}\subset \Lambda^{0,L;\xi}_{\tau,\hat\tau}$ such that $\cup_{j\in\bN^+} \left( \Lambda^{0,L;\xi}_{\tau,\hat\tau} \cap B_{\delta/3}(x^j)\right)=\Lambda^{0,L;\xi}_{\tau,\hat\tau}$. Set $D_1=B_{\delta/3}(x^1)\cap \Lambda^{0,L;\xi}_{\tau,\hat\tau}$, and 
$$D_j=\left( B_{\delta/3}(x^j)-(\cup_{i=1}^{j-1} B_{\delta/3}(x^i))\right)\cap \Lambda^{0,L;\xi}_{\tau,\hat\tau},\quad \text{for } j>1.$$
 Then $\{D^{j}\}_{j\in\bN^+}$ is a partition of $\Lambda^{0,L;\xi}_{\tau,\hat\tau}$ with diameter diam$(D^j)<\delta$, i.e., $D^j\subset \Lambda^{0,L;\xi}_{\tau,\hat\tau}$, $\cup_{j\in\bN^+}D^j=\Lambda^{0,L;\xi}_{\tau,\hat\tau}$, $D^i\cap D^j=\emptyset$ if $i\neq j$, and for any $x,y\in D^j$, $\|x-y\|_0<\delta$. 
 
Then the rest of the proof is similar to that of  \cite[Theorem 3.4]{qiu2017viscosity}.  For each $j\in\bN^+$, take $\bar x^j\in D^j$, and a straightforward application of Proposition \ref{prop-value-func} (i) leads to some $\theta^j\in\cU$ satisfying 
$$0\leq J(\hat\tau,\bar x^j;\theta^j) -V(\hat\tau,\bar x^j):= \alpha^j\quad \text{a.s., with } E|\alpha^j|<\frac{\eps}{2^j}.$$
Thus, for each $x\in D^j$, it holds that
\begin{align*}
&J(\hat\tau,x;\theta^j) -V(\hat\tau,x)\\
&\leq 
	|J(\hat\tau,x;\theta^j) -J(\hat\tau,\bar x^j;\theta^j)|
	+|J(\hat\tau,\bar x^j;\theta^j)-V(\hat\tau,\bar x^j)|
	+|V(\hat\tau,\bar x^j) -V(\hat\tau,x)|
\\
&\leq 2\,\eps+\alpha^j,\quad \text{a.s.}
\end{align*}

In view of Assumption  $({\mathcal A} 1)$ (iii), we observe that for any $\theta\in\cU$,  $X^{\tau,\xi;\theta}_{\hat\tau}$ is almost surely valued in $\Lambda^{0,L;\xi}_{\tau,\hat\tau}$. For each $\theta\in\cU$, put
\[  \tilde{\theta}(s)=
\begin{cases}
\theta(s), \quad & \text{if }s\in[0,\hat\tau);\\
\sum_{j\in\bN^+} \theta^j(s)1_{D^j}(X^{\tau,\xi;\theta}_{\hat\tau}),\quad &\text{if }s\in [\hat\tau,T].\\
\end{cases}
\]
Then it follows that
\begin{align*}
V(\tau,\xi)
&\leq J(\tau,\xi;\tilde\theta)\\
&=	E_{\sF_{\tau}}\left[
	\int_{\tau}^{\hat \tau}  
		f\left(s,X_s^{\tau,\xi;\theta},\theta(s)\right)\,ds
	+J\left(\hat\tau,X_{\hat\tau}^{\tau,\xi;\theta};\tilde\theta\right)
	\right]\\
&\leq
	E_{\sF_{\tau}}\left[
	\int_{\tau}^{\hat \tau}  
		f\left(s,X_s^{\tau,\xi;\theta},\theta(s)\right)\,ds
		+V\left(\hat\tau,X_{\hat\tau}^{\tau,\xi;\theta}\right)
		+\sum_{j\in\bN^+}\alpha^j
	\right]+2\,\eps,
\end{align*}
where $\{\alpha^j\}$ is independent of the choices of $\theta$. Taking infimums and then expectations on both sides, we arrive at
%\begin{align*}
$EV(\tau,\xi) \leq
E \overline V(\tau,\xi) + 3\,\eps$.
%\end{align*}
By the arbitrariness of $\eps>0$, we have $E\overline V(\tau,\xi)\geq  EV(\tau,\xi)$, which together with the obvious relation $\overline V(\tau,\xi)\leq V(\tau,\xi)$ yields
that $\overline V(\tau,\xi)=V(\tau,\xi)$ a.s.
\end{proof}

\end{appendix}
%\section{Comments}
\bibliographystyle{siam}
%\bibliography{ref_qjn}

\end{document}